\documentclass[10pt]{amsart}
\usepackage{graphicx} 
\usepackage[a4paper, margin=2.9 cm]{geometry}
\usepackage{hyperref}
\usepackage{amsthm}
\usepackage{complexity}
\usepackage{thm-restate}
\usepackage{amssymb}
\usepackage{amsmath}
\usepackage{mathtools} 
\usepackage{cleveref}
\usepackage{comment}
\usepackage{setspace}
\usepackage{xcolor}
\usepackage[shortlabels]{enumitem}
\usepackage{listings}
\usepackage{blkarray}
\usepackage{lmodern}
\usepackage[english]{babel}
\usepackage{float}
\usepackage{tikz}
\usetikzlibrary{arrows.meta,arrows}
\usetikzlibrary{arrows,decorations.markings}
\usetikzlibrary{decorations.pathmorphing}
\usepackage[square,sort,comma,numbers]{natbib}
\usepackage{thm-restate}

\newenvironment{claimproof}[1][\proofname]{\proof[#1]}{\endproof}
\newtheorem{theorem}{Theorem}[section]
\newtheorem{lemma}[theorem]{Lemma}
\newtheorem{corollary}[theorem]{Corollary}

\newtheorem{observation}[theorem]{Observation}
\newtheorem{claim}[theorem]{Claim}
\newtheorem*{claim*}{Claim}
\newtheorem{proposition}[theorem]{Proposition}

\newtheorem{problem}[theorem]{Problem}

\theoremstyle{definition}

\newtheorem*{definition*}{Definition}

\newcommand{\cC}{\mathcal{C}}
\newcommand{\calG}{\mathcal{G}}

\newcommand{\mbE}{\mathbb{E}}
\newcommand{\RR}{\mathbb{R}}
\newcommand{\eps}{\varepsilon}

\newcommand{\Bin}{\mathrm{Bin}}

\newcommand{\calD}{\mathcal{D}}
\renewcommand{\Pr}{\mathbb{P}}

\newcommand{\ceil}[1]{
    \left\lceil #1 \right\rceil
}

\title[Coloring small locally sparse degenerate graphs]{Coloring small locally sparse degenerate graphs \\ and related problems}

\author
{
Domagoj Brada\v{c}
}
\author{Jacob Fox}
\thanks{The second author gratefully acknowledges funding from NSF awards DMS-2452737 and DMS-2154129. }
\author 
{
Raphael Steiner
}
\thanks{The third author gratefully acknowledges funding from the SNSF Ambizione Grant No. 216071. }
\author
{
Benny Sudakov 
}\thanks{The fourth author gratefully acknowledges funding from the SNSF Grant 200021-228014. }
\author{
Shengtong Zhang
}

\address[Brada\v{c}]{Department of Mathematics, EPFL, Lausanne, Switzerland.} 
\email{\tt domagoj.bradac@epfl.ch}
\address[Steiner, Sudakov]{Department of Mathematics, Institute for Operations Research, ETH Z\"{u}rich, Switzerland.}
\email{\tt $\{$raphaelmario.steiner, benjamin.sudakov$\}$@math.ethz.ch}
\address[Fox, Zhang]{Department of Mathematics, Stanford University, Stanford, CA.}
\email{\tt $\{$jacob.fox,stzh1555$\}$@stanford.edu}

\date{}

\begin{document}


\begin{abstract}
    \setlength{\parskip}{\smallskipamount}
    \setlength{\parindent}{0pt}
    \noindent
The classic upper bound on the chromatic number of $d$-degenerate graphs is $d+1$, shown to be tight by complete graphs. A natural question is whether this bound remains tight if further constraints are imposed on the graph, such as forbidding large cliques. Quite surprisingly, classic constructions of Tutte and Zykov from the early 50s and a generalization due to Kostochka and Ne\v{s}et\v{r}il from 1999 show that there exist $d$-degenerate $(d+1)$-chromatic graphs that are triangle-free (and even have large girth). However, the orders of all known constructions of $d$-degenerate $(d+1)$-chromatic graphs with small clique number grow rapidly with~$d$. Motivated by this phenomenon and addressing a problem posed by the second author at the Oberwolfach Graph Theory workshop,
we prove that the minimum order $f(d)$ of a $d$-degenerate triangle-free graph of chromatic number $d+1$ grows exponentially in a polynomial of $d$, concretely $e^{\Omega(d)}\le f(d)\le e^{O(d^2\log d)}.$ The lower bound follows from a novel upper bound on the chromatic number of triangle-free graphs that may be of independent interest: Every triangle-free $d$-degenerate graph $G$ on $n \le e^{O(d)}$ vertices satisfies $$\chi(G)\le O\left(\frac{d}{\log\left(d/\log n\right)}\right).$$ We also extend this to a more general result about degenerate graphs with sparse neighborhoods, which has applications to a variety of graph coloring problems: For example, we prove that every counterexample to Hadwiger's conjecture with parameter $t$ must have a complete bipartite subgraph with one exponentially large side ($K_{a,b}$ where $a=(\log t)^{1/2-o(1)}$ and $b=e^{t^{1-o(1)}}$) or a small and very dense subgraph (of order $\le t$ with $\ge t^{2-o(1)}$ edges) in some neighborhood.

For the upper bound on $f(d)$ we establish a surprising connection between $f(d)$ and the on-line-chromatic number $g(n)$ of $n$-vertex triangle-free graphs. We also give an asymptotic improvement of the previous best upper bound for $g(n)$ due to Lov\'{a}sz, Saks and Trotter from 1989.

Along the way, we disprove a generalization of Harris' conjecture on fractional coloring to graphs of bounded clique number and raise numerous problems which open up various interesting directions to explore for future research.
\end{abstract}

\maketitle

\section{Introduction
}\label{sec:intro}
The \emph{chromatic number} $\chi(G)$ of a graph $G$ is one of the most important graph parameters and has been actively studied for about a century, see~\cite{scott22} for a recent survey. Despite this, we still do not fully understand what makes the chromatic number of a graph large, as illustrated by the fact that it is \NP-hard to compute $\chi(G)$~\cite{karp} and even intractable to approximate it to within a factor $n^{1-\varepsilon}$ for any $\varepsilon>0$, where $n$ is the order of the graph~\cite{Feige}. Consequently, to bound the chromatic number of a given graph,  one often has to resort to one of a few known bounds that are easy to compute, such as the maximum degree or the degeneracy of the graph. Recall that a graph $G$ is called $d$-degenerate for some $d\in \mathbb{R}_{\ge 0}$ if its vertex set can be linearly ordered as $v_1,\ldots,v_n$ such that for each $i\in \{1,\ldots,n\}$, vertex $v_i$ has at most $d$ neighbors among the previous vertices $\{v_1,\ldots,v_{i-1}\}$ in the ordering. Greedy coloring along this ordering shows that $\chi(G)\le d+1$ for every $d$-degenerate graph, which is an improvement of the weaker bound $\chi(G)\le \Delta+1$ in terms of the maximum degree $\Delta$ of the graph. Both of these bounds are tight, as evidenced by the complete graphs $K_{d+1}$ and $K_{\Delta+1}$. But what if we forbid large complete subgraphs? Can we then improve these upper bounds? A seminal result of Johansson~\cite{johansson,johansson2} states that indeed a significant saving can be obtained when considering the maximum degree. Namely, a graph of maximum degree $\Delta$ has chromatic number at most $O\left(\frac{\Delta}{\log \Delta}\right)$ if it is triangle-free and at most $O_r\left(\frac{\Delta\log\log\Delta}{\log \Delta}\right)$ if it is $K_r$-free for $r\ge 4$; see also Molloy~\cite{molloy} for improved constants.

However, quite surprisingly, when moving from maximum degree to degeneracy, such a saving is no longer possible: Already in the early 50s, Zykov~\cite{zykov} and Tutte~\cite{descartes} published constructions which show that for every $d\in \mathbb{N}$, there exists a triangle-free $d$-degenerate graph of chromatic number $d+1$, so even forbidding triangles does not achieve any improvement of the degeneracy bound. As shown by Kostochka and Ne\v{s}et\v{r}il~\cite{kostnest} in 1999, one can modify Tutte's construction to even obtain $d$-degenerate $(d+1)$-chromatic graphs of large girth. 

This may seem like a conclusive answer to our above question, but this is not quite the case. Indeed, one can check that the number of vertices of the $d$-degenerate graphs constructed by Zykov and Tutte grows extremely rapidly (much faster than exponential in $d$). It is thus natural to ask whether this property of the constructions is a necessary one: If a triangle-free $d$-degenerate graph has not too many vertices, can we then improve the upper bound $d+1$ on the chromatic number significantly? Moreover, if so, how fast does the smallest number of vertices $f(d)$ of a triangle-free $d$-degenerate graph with chromatic number $d+1$ grow?  These questions, which were also explicitly posed by the second author at the Oberwolfach Graph Theory Workshop in January 2025 and came from work on a generalization of Ramsey numbers \cite{FoxTidorZhang}, are the topic of our paper. 

\subsection{Main results.}

\paragraph*{\textbf{Coloring small locally sparse degenerate graphs.}} Our first theorem gives a positive answer to the first question above, as follows.

\begin{theorem}\label{thm:main2}
    There are absolute constants $c,C >0$ such that the following holds. Let $d\in \mathbb{R}_{\ge 1}$ and let $G$ be a triangle-free $d$-degenerate graph on $n \le 2^{cd}$ vertices. Then \[ \chi(G) \le \frac{C d}{\log (d / \log n)}. \]
\end{theorem}

 Alon, Krivelevich and the fourth author~\cite{AKS} generalized Johansson's aforementioned result~\cite{johansson} on the chromatic number of triangle-free graphs with maximum degree $\Delta$ to the more general setting of graphs with sparse neighborhoods. Namely, they showed that a graph of maximum degree $\Delta$ in which every neighborhood spans at most $\Delta^2/f$ edges has chromatic number at most $O\left(\frac{\Delta}{\log f}\right)$. Given this, it is natural to wonder whether we can also extend Theorem~\ref{thm:main2} to a more general setting of sparse neighborhoods and still obtain an improvement over the bound $d+1$. Our next result provides such an extension. Given a graph with a linear order on its vertices, the \emph{left-neighborhood} of a vertex is the set of neighbors that precede it in the linear order. 

\begin{theorem}\label{thm:sparse}
      There exist absolute constants $C, c > 0$ such that the following holds. Let $d\in \mathbb{R}_{\ge 1}$, let $2<f<d^2$ and let $G$ be a graph with an ordering such that every left-neighborhood spans at most $d$ vertices and at most $d^2/f$ edges. If $G$ has $n \le 2^{cd}$ vertices, then \[ \chi(G) \le \frac{C d}{\min\{\log  f,\log(d/\log n)\}}. \]  
\end{theorem}

Note that Theorem~\ref{thm:main2} can be recovered as a special case of Theorem~\ref{thm:sparse} when $f=\Theta(d^2)$. 

\paragraph*{\textbf{Small triangle-free $d$-degenerate $(d+1)$-chromatic graphs via on-line coloring.}} To produce examples of small triangle-free $d$-degenerate graphs with chromatic number $d+1$, we establish a surprising connection between this problem and the \emph{on-line chromatic number} of triangle-free graphs of order $n$. An on-line graph is a tuple $(G, \prec)$, where $G$ is a graph and $\prec$ is a linear order on the vertex set, inducing an ordering $v_1 \prec v_2 \prec \dots \prec v_n$. For a hereditary graph class $\calG$, an on-line coloring algorithm $A$ for $\calG$ is first given as input a positive integer $n$, representing the total number of vertices. Then, the algorithm is presented vertices $v_1, \dots, v_n$ of an on-line graph $(G, \prec),$ where $G \in \calG$ in order, one at a time. Having been presented vertex $v_i,$ the algorithm assigns it a color from a color palette $\cC_n$ distinct from the colors previously assigned to its neighbors. 
For more information on on-line coloring, we refer to the survey of Kierstead~\cite{kierstead}. The on-line chromatic number of $n$-vertex graphs from $\calG$ is the minimum size of a color palette $\cC_n$ such that there is an on-line algorithm using $\cC_n$ as described above.


For positive integers $r,n$, let $g_r(n)$ be the on-line chromatic number of the family of $K_r$-free graphs on $n$ vertices. The following statement makes precise the connection between on-line chromatic numbers and the existence of small locally sparse $d$-degenerate and $(d+1)$-chromatic graphs.

\begin{restatable}{proposition}{propobsrestate} \label{prop:obs}
    Given $d,r\in \mathbb{N}$, let $N\in \mathbb{N}$ be the smallest number satisfying $g_r(N)\ge d+1$. Then there exists a $K_r$-free $d$-degenerate graph with chromatic number $d+1$ on at most $\frac{d^{N}-1}{d-1}$ vertices.  
\end{restatable}

The previous best known asymptotic bounds on $g_r(n)$ are $\Omega(n^{(r-2)/(r-1)})\le g_r(n)\le O\left(\frac{n}{\log^{(2r-4)} n}\right)$, where $\log^{(b)}(\cdot)$ denotes the $b$-fold iterated logarithm function, defined as $\log^{(b)}(x):=\log(\log^{b-1}(x))$ for $b\ge 2$. The above lower bound for $r=3$ is due to a well-known and relatively simple construction, see e.g.~\cite{dvorakkawab,gutow,kierstead}, which can be generalized to all $r\ge 3$. This was probably known before, but we could not find an explicit reference in the literature and thus shall include a proof for completeness in Section~\ref{sec:online}. The stated upper bound on $g_r(n)$ is due to Lov\'{a}sz, Saks and Trotter~\cite{lovasz}. While they state the result for on-line coloring graphs of chromatic number less than $r$, their coloring strategy applies more generally to $K_r$-free graphs.

Our next result yields the first improvement of the above mentioned upper bound on $g_r(n)$ due to Lov\'{a}sz, Saks and Trotter from 1989.
\begin{theorem}\label{thm:jacob}
Let $r\ge 3$. The maximum on-line chromatic number of $K_r$-free $n$-vertex graphs satisfies 
$$g_r(n)\le O_r\left(\frac{n \log^{(r-1)} n}{\log^{(r-2)} n}\right).$$
\end{theorem}
For example, for $r=3$ this improves the bound $g_3(n)\le O(n/\log\log n)$ of Lov\'{a}sz, Saks and Trotter roughly by a $\log n$-factor to $g_3(n)\le O(n\log\log n/\log n)$.

\paragraph*{\textbf{Smallest triangle-free $d$-degenerate $(d+1)$-chromatic graphs.}} As a consequence of the previously presented results, we determine the growth of the minimum order $f(d)$ of a triangle-free $d$-degenerate $(d+1)$-chromatic graph (for $d\in \mathbb{N}$) to be exponential in a polynomial of $d$, thus addressing the second of the two questions mentioned in the introduction.

\begin{theorem}\label{thm:main1}
    It holds that $e^{\Omega(d)}\le f(d)\le e^{O(d^2\log d)}$.
\end{theorem}

Indeed, the lower bound is a direct consequence of Theorem~\ref{thm:main2}, while the upper bound follows directly from Proposition~\ref{prop:obs}, using the above mentioned bound that $g_3(n)\ge \Omega(\sqrt{n})$. 

\paragraph*{\textbf{Extending to $K_r$-free graphs.}} A natural question is whether Theorem~\ref{thm:main1} can be extended to $K_r$-free graphs for $r\ge 4$. For $d\in \mathbb{N}$ let $f_r(d)$ denote the smallest order of a $K_r$-free $d$-degenerate, $(d+1)$-chromatic graph. The lower bound $g_r(n)\ge \Omega(n^{(r-2)/(r-1)})$ in conjunction with Proposition~\ref{prop:obs} yields that $f_r(d)\le e^{O(d^{(r-1)/(r-2)}\log d)}$ for every fixed $r$. Whether $f_r(d)$ is at least exponential in $d$ remains open. To generalize our lower bound on $f(d)=f_3(d)$ in Theorem~\ref{thm:main1} to $f_r(d)$ one would need to prove an analogue of Theorem~\ref{thm:main2} for $K_r$-free graphs. While such a result may be true, it seems difficult to achieve this with our proof method. Our coloring strategy for Theorem~\ref{thm:main2} is based on repeated sampling of independent sets from a probability distribution that is a modification of a distribution recently introduced by Martinsson~\cite{martinsson} to settle a conjecture of Harris~\cite{harris} on fractional coloring. Harris' conjecture~\cite{harris} stated that the \emph{fractional} chromatic number of every triangle-free $d$-degenerate graph $G$ is at most $O(d/\log d)$. The fractional chromatic number $\chi_f(G)$ of a graph $G$ is defined as $p^{-1}$, where $p\in (0,1]$ is maximal such that there exists a probability distribution on independent sets of $G$ such that every vertex is contained in a randomly drawn independent set with probability at least $p$.  Thus, Harris' conjecture is equivalent to the existence of a probability distribution on the independent sets of a triangle-free $d$-degenerate graph $G$ such that every vertex of $G$ is contained in a random independent set sampled from the distribution with probability $\Omega(\log d/d)$.  Martinsson~\cite{martinsson} achieved this using a novel and elegant random process. Our next result below shows that, quite surprisingly, this upper bound on the fractional chromatic number does not extend to $K_r$-free $d$-degenerate graphs for any $r\ge 4$. This presents a non-trivial obstacle for extending our proof method to the $K_r$-free case. 

\begin{restatable}{theorem}{thmlowerrestate}
    \label{thm:lower}
    For each integer $r \geq 3$ and sufficiently large $d\in \mathbb{R}_{\ge 1}$, there exists a $d$-degenerate, $K_r$-free graph $G$ with $\chi_f(G) \ge \Omega_r\left(\frac{d}{\log^{(r - 2)} d}\right)$. In other words, for every probability distribution on the independent sets of $G$, some vertex is contained in an independent set drawn from the distribution with probability at most $O_r\left(\frac{\log^{(r - 2)} d}{d}\right)$.
\end{restatable}

\subsection{Applications.} 

\paragraph{While Theorem~\ref{thm:sparse} is a natural result in its own right, it also has several interesting applications to various graph coloring problems on sparse graphs. Here we present some preliminary examples of such applications, however we expect many more applications to be found in the future.}

\paragraph*{\textbf{Hadwiger's conjecture.}} Hadwiger's conjecture from 1943~\cite{hadwiger} is perhaps the most famous open graph coloring problem, which states that $K_t$-minor-free graphs have chromatic number at most $t-1$. Even for restricted families of graphs, this remains widely open. In particular, Hadwiger's conjecture for triangle-free graphs is a known open problem proposed and studied by various researchers~\cite{bucic,delcourt,dvorak,norin}. The currently best known upper bound on the chromatic number of triangle-free $K_t$-minor-free graphs in general is $Ct$ for some large constant $C$ due to Delcourt and Postle~\cite{delcourt}, and the best lower bound is $\lceil (t+1)/2\rceil$ by Dvo\v{r}\'{a}k and Kawarabayashi~\cite{dvorakkawab}.
    
Combining Theorem~\ref{thm:main2} with results of Dujmovi\'{c} et al.~\cite{dujmovic} and a recent breakthrough of Gorsky, Seweryn and Wiederrecht~\cite{gowi} on the Graph Minor Structure Theorem, we can prove the following.
    \begin{corollary}\label{cor:had1} For every $\varepsilon>0$ there exists $K>0$ such that the following holds for every $t\in \mathbb{N}$. Every triangle-free graph $G$ with no $K_t$-minor and $\Delta(G)\le e^{t^{1-\varepsilon}}$ satisfies $$\chi(G)\le \frac{Kt}{\sqrt{\log t}}.$$ 
    \end{corollary}
    Thus triangle-free graphs satisfy Hadwiger's conjecture with an asymptotically improved bound, unless they have vertices of huge degree. This result also generalizes with the same asymptotic bounds to $H$-free graphs where $H$ is any almost bipartite graph (i.e., can be made bipartite by deleting a vertex).

    Another interesting consequence is the following novel structural result about counterexamples to Hadwiger's conjecture, stating that they either have complete bipartite subgraphs with one exponentially large side or very small and dense subgraphs in the neighborhood of some vertex.
\begin{corollary}\label{cor:had2}
Let $t\in \mathbb{N}$ and let $G$ be any counterexample to Hadwiger's conjecture for parameter~$t$. Then $G$ has a $K_{a,b}$-subgraph where $a=(\log t)^{1/2-o(1)}$ and $b=e^{t^{1-o(1)}}$, or a subgraph $H\subseteq G$ with $|V(H)|\le t$ and $|E(H)|\ge t^{2-o(1)}$ which is contained in the neighborhood of some vertex.
    \end{corollary}

\paragraph*{\textbf{Ramsey numbers of graphs with a given chromatic number.}} A well-known conjecture of Erd\H{o}s (cf.~Section~2.4.2 in~\cite{chunggraham}) states that there exists an absolute constant $c>0$ such that every graph $G$ with $\chi(G)=k$ satisfies $r(G)\ge c\cdot r(K_k)$ (here $r(G)$ denotes the graph Ramsey number). Using Theorem~\ref{thm:main2}, we can obtain the following structural result about potential counterexamples to Erd\H{o}s' conjecture: They either have an exponential number of vertices or a subgraph comparable in size and density to a $K_k$.
    \begin{corollary}\label{cor:erd}
    If a graph $G$ with $\chi(G)=k$ satisfies $r(G)<r(K_k)$, then $|V(G)|\ge e^{\Omega(k)}$ or $G$ contains a subgraph on $O(k)$ vertices with at least $\Omega(k^2)$ edges.
    \end{corollary}
\paragraph*{\textbf{Coloring hereditary classes below the degeneracy bound.}} 
The following is a consequence of Theorem~\ref{thm:sparse}.
\begin{corollary}\label{cor:logbound}
For every almost bipartite graph $H$ and every $\varepsilon>0$ there exists $K=K(H,\varepsilon)>0$ such that every $d$-degenerate $H$-free graph $G$ on $n$ vertices satisfies $$\chi(G) \le K \max\left(\frac{d}{\log d},\log^{1+\varepsilon} n\right).$$
\end{corollary}
Interestingly, this result allows us to show that degeneracy-based upper bounds for the chromatic number of graphs in hereditary classes can almost always be improved by a logarithmic factor. To be precise, consider any hereditary graph class $\mathcal{G}$ (i.e., closed under taking induced subgraphs) and let $d(n)$ denote the maximum of the average degree among $n$-vertex graphs in the class. If $d(\cdot)$ is monotone, the degeneracy bound yields that the chromatic number of $n$-vertex graphs in $\mathcal{G}$ is at most $d(n)+1$. Using Corollary~\ref{cor:logbound} we obtain an asymptotic improvement over this degeneracy bound in almost the full regime of possible functions $d(n)$.
\begin{corollary}\label{cor:hereditary}
    Let $\mathcal{G}$ be a hereditary graph class and let $d:\mathbb{N}\rightarrow \mathbb{R}_{\ge 1}$ be a monotone function satisfying $\log^{1+\varepsilon} n\le d(n)=o(n)$, such that every $n$-vertex graph $G\in \mathcal{G}$ has average degree at most $d(n)$. Then every $n$-vertex graph $G\in \mathcal{G}$ satisfies $\chi(G)\le O\left(\frac{d(n)}{\log d(n)}\right)$.
\end{corollary}
There are several graph classes for which the best upper bound $d(n)$ on the maximal average degree lies in the range between $\log^{1+\varepsilon} n$ and $o(n)$, and to which Corollary~\ref{cor:hereditary} is thus applicable (for example graphs with a finite family of forbidden bipartite subgraphs containing a cycle and certain graph families defined by forbidden edge ordered subgraphs~\cite{gerbner,janzer}).
We remark that the bound given by Corollary~\ref{cor:hereditary} is actually tight up to constant factors: Let $\alpha\in (0,1)$ be any given constant and let us consider the graph class $$\mathcal{G}_\alpha:=\left\{G \text{ is a graph}\bigg\vert e(H)\le \frac{100}{1-\alpha}|V(H)|^{\alpha+1}\ \text{ for all }H\subseteq G\right\}.$$ Clearly, $\mathcal{G}_\alpha$ is hereditary, and the function $d(n):=\frac{200}{1-\alpha} n^\alpha$ forms an upper bound on the maximum of the average degrees of $n$-vertex graphs in $\mathcal{G}_\alpha$. On the other hand, it is not difficult to check that with high probability as $n\rightarrow \infty$, a random graph $G(n,p)$ with $p=n^{\alpha-1}$ is a member of $G_\alpha$ and has chromatic number $\Theta\left(\frac{np}{\log(np)}\right)=\Theta\left(\frac{d(n)}{\log(d(n))}\right)$.

\paragraph*{\textbf{Coloring beyond planar graphs.}} Theorem~\ref{thm:main2} also has applications to bounds on the chromatic number of beyond planar graphs with local restrictions such as triangle-freeness. Here we give one such example. For integers $g,k\in \mathbb{N}$, a graph is called \emph{$(g,k)$-planar} if it can be drawn on a surface of Euler genus $g$ with at most $k$ crossings per edge. In the special case $g=0$, this is the well-studied class of $k$-planar graphs. Ossona de Mendez, Oum and Wood~\cite{demendez} proved that $(g,k)$-planar graphs are $O(\sqrt{(g+1)(k+1)})$-degenerate, which also implies that their chromatic number is $O(\sqrt{(g+1)(k+1)})$ (see also~\cite{wood}). Combining Theorem~\ref{thm:main2} with a result on clustered coloring by Wood~\cite{wood}, we obtain an asymptotically improved bound for triangle-free graphs.
    \begin{corollary}\label{cor:gk}
   Triangle-free $(g,k)$-planar graphs have chromatic number $O\left(\frac{\sqrt{(g+1)(k+1)}}{\log ((g+1)(k+1))}\right)$.
    \end{corollary}
    As an example, this shows that triangle-free $k$-planar graphs have chromatic number at most $O(\frac{\sqrt{k}}{\log k})$, improving over the best possible $\Theta(\sqrt{k})$ bound for the chromatic number of general $k$-planar graphs~\cite{pach}. Corollary~\ref{cor:gk} generalizes, with the same asymptotic bounds, to $H$-free $(g,k)$-planar graphs where $H$ is any almost bipartite graph.

\medskip

\paragraph*{\textbf{Notation and Terminology.}}
Throughout the manuscript, we denote by $\exp(\cdot)$ the exponential function and by $\log(\cdot)$ the \emph{natural} logarithm. For a natural number $n$, we denote $[n]:=\{1,\ldots,n\}$. Given a graph $G$ equipped with a linear ordering $v_1,\ldots,v_n$ of its vertices, for every $i\in [n]$ we denote by $N_L(v_i):=N_G(v_i)\cap \{v_1,\ldots,v_{i-1}\}$  and $N_R(v_i):=N_G(v_i)\cap \{v_{i+1},\ldots,v_n\}$ the \emph{left and right neighborhood of $v_i$}, respectively. Given a subset $U$ of vertices of a graph, we denote by $G[U]$ the induced subgraph of $G$ with vertex set $U$. 

\section{Proofs of Theorems~\ref{thm:main2} and~\ref{thm:sparse}}
The proofs of Theorems~\ref{thm:main2} and~\ref{thm:sparse} will rely on repeated sampling of independent sets from a modification of the probability distribution on independent sets introduced by Martinsson. Martinsson's random process iteratively updates non-negative real-valued weights on vertices and uses those to guide the choice of the random independent set. The modification we use here allows for a better control on the growth of these weights, in particular our process ``caps'' weights if they grow to large. This is important for technical reasons in our applications of this distribution, for example it enables us to effectively use Azuma's inequality in our proofs of Theorems~\ref{thm:main2} and~\ref{thm:sparse}.

The following statement establishes a distribution on the independent sets of a triangle-free graph that has the suitable properties necessary for our later arguments.

\begin{lemma} \label{lem:main}
    Let $d\in \mathbb{R}_{\ge 1}$, $0 < \alpha < \frac{\log d}{4d}$ and $G$ be a $d$-degenerate triangle-free graph. Let $v_1,\ldots,v_n$ be a linear ordering witnessing that $G$ is $d$-degenerate, i.e., such that $|N_L(v_i)|\le d$ for all $i\in [n]$. Then there exists a probability distribution $\mathcal{D}$ on the independent sets of $G$ with the following properties.

    \begin{enumerate}[label=\alph*)]
        \item \label{prop:large-prob} For every $i \in [n], \, \Pr_{I \sim \calD}[v_i \in I] \ge \alpha / 4.$
        \item \label{prop:prob-not-too-big} For every $i \in [n]$ and every independent set $S \subseteq \{v_1, \dots, v_{i-1}\},$ it holds that
        \[ \Pr_{I \sim \calD}[v_i \in I \, \vert \, I \cap \{v_1, \dots, v_{i-1}\} = S] \le \alpha \cdot e^{2 \alpha d}. \]
        \item \label{prop:indepofwhatcomesafter} For every $i\in [n]$, the distribution of the set $I\cap \{v_1,\ldots,v_i\}$ where $I$ is sampled from $\mathcal{D}$ depends only on the ordered induced subgraph $G[\{v_1,\ldots,v_{i}\}]$ of $G$.
    \end{enumerate}
\end{lemma}
\begin{proof}
    We consider the following random process to generate a random independent set $I$ in $G$, which is a modification of the aforementioned process by Martinsson~\cite{martinsson}.

    The process keeps track of a vertex weighting $w:V(G)\rightarrow \mathbb{R}_{\ge 0}$. The weights are initialized as $w(v_i)\leftarrow \alpha$ for every $i\in [d]$. The process then iterates through the numbers $i=1,\ldots,n$ in order, while updating the weights of the vertices and creating the independent set $I$, as follows:
    \begin{itemize}
        \item If $w(v_i) > \alpha \cdot e^{2 \alpha d}$, then set $w(v_i) \leftarrow 0$. Otherwise, keep its value unchanged.
        \item Next, with probability $1-e^{-w(v_i)}$ add $v_i$ to $I$ and set $w(v_j) \leftarrow 0$ for every $v_j\in N_R(v_i)$. Alternatively (with probability $e^{-w(v_i)}$), do not put $v_i$ into $I$ and set $w(v_j) \leftarrow w(v_j) \cdot e^{w(v_i)}$ for every $v_j\in N_R(v_i)$. 
    \end{itemize}

Notice that by definition of the process, the outcome $I$ is indeed an independent set with probability~$1$, since once a vertex $v_i$ is put into $I$ all its right-neighbors will have weight $0$ throughout the rest of the process, and hence will not be put into $I$ with probability $e^{0}=1$. 

Hence, the above process gives rise to a well-defined probability distribution $\mathcal{D}$ on the independent sets of $G$. We start by proving Property~\ref{prop:large-prob} of the distribution in the lemma, i.e., that $\mathbb{P}[v_k \in I]\ge \alpha / 4$ for every $k\in [n]$. So, fix $k\in [n]$ in the following. Similar to the analysis of Martinsson~\cite{martinsson}, to give the desired lower bound on $\mathbb{P}[v_k\in I]$ we will compare the above process to a slightly modified process to generate a weight function $\tilde{w}:V(G)\rightarrow \mathbb{R}_{\ge 0}$, as follows:

As before, initialize the weights as $\tilde{w}(v_i)\leftarrow \alpha$ for $i=1,\ldots,n$. Then, for $i=1,\ldots,n$, do the following.
\begin{itemize}
    \item If $\tilde{w}(v_i) > \alpha \cdot e^{2 \alpha d}$, then set $\tilde{w}_i(v_i) \leftarrow 0$. Then proceed as follows:
    \item If $v_i\notin N_L(v_k)$, then update the weights as in the previously described process for $w$ at step~$i$.
    \item If $v_i\in N_L(v_k)$, then we always perform updates corresponding to not including $v_i$ in the independent set. That is, for every $v_j\in N_R(v_i)$ update its weight as $\tilde{w}(v_j)\leftarrow \tilde{w}(v_j)\cdot e^{\tilde{w}(v_i)}$.
\end{itemize}

Note that the only difference between the above process for $\tilde{w}$ and the previous process for $w$ is that when considering a left-neighbor $v_i$ of $v_k$ of current weight $x\le \alpha e^{2 \alpha d}$, then we \emph{certainly} proceed by multiplying the weights of all its right neighbors by $e^{x}$, while we only do this with probability $e^{-x}$ in the previous process for $w$. Using this relationship, we can easily observe the following auxiliary claim (analogous to the corresponding claim in the paper by Martinsson~\cite{martinsson}). In the following, for any $i, j \in [n]$ we denote by $w_i(v_j)$ and $\tilde{w}_i(v_j)$, respectively, the random variables corresponding to the weights of vertex $v_j$ in the first, respectively second, random process after step $i$ (this includes $i=0$, when $w_0(v_j)=\tilde{w}_0(v_j)=\alpha$ for every $j\in [n]$). In the following, we consider the random variable $X:=\sum_{v_j\in N_L(v_k)}\tilde{w}_{k-1}(v_j)$.  
\begin{claim} \label{claim:f(w)}
For every function $f:[0,\infty)\rightarrow [0,\infty)$ with $f(0)=0$, we have 
$$\mathbb{E}[f(w_{k-1}(v_k))]= \mathbb{E}[f(\tilde{w}_{k-1}(v_k))e^{-X}].$$
\end{claim}
\begin{claimproof}
    The proof is analogous to the proof of Claim~2.2 in~\cite{martinsson}. The first $k-1$ steps of the process corresponding to the weights $w$ may be represented by a sequence $\mathbf{a} = (a_1, \dots, a_{k-1}),$ with each $a_i \in \{1, 2\},$ where $a_i = 1$ represents putting $v_i$ into the independent set $I$ and setting all the weights of $N_R(v_i)$ to $0$ and $a_i = 2$ represents not putting $v_i$ into $I$ and multiplying the weights of each vertex in $N_R(v_i)$ by $e^{w_i(v_i)}$. Analogously, let $\tilde{\mathbf{a}} \in \{1, 2\}^{k-1}$ denote the sequence of the process corresponding to weight $\tilde{w}$ and note that, by definition, we have $\tilde{\mathbf{a}}_i = 2$ for all $i$ such that $v_i \in N_L(v_k)$.

    Call a sequence $\mathbf{s} \in \{1,2\}^{k-1}$ valid if $s_i = 2$ for all $i$ such that $v_i \in N_L(v_k).$ Observe that if $\mathbf{s}$ is not valid, then $w_{k-1}(v_k) = 0,$ so such a sequence does not contribute to $\mbE[f(w_{k-1}(v_k))].$ Note that for any valid sequence $\mathbf{s}$, if $\mathbf{a} = \tilde{\mathbf{a}}=\mathbf{s},$ the weights $w$ and $\tilde{w}$ are the same throughout the entire process.     For a given valid sequence $\mathbf{s} \in \{1, 2\}^{k-1}$ and $i,j \in [k-1]$ let us denote by $w^{\mathbf{s}}_i(v_j)$ the weight of vertex $v_j$ after stage $i$ when making choices according to the sequence $\mathbf{s}$ (this definition applies to both processes $w$ and $\tilde{w}$, since they agree when $\mathbf{s}$ is valid). Recall that the weights $w(v_i)$ and $\tilde{w}(v_i)$ are not changed in either process after step $i.$ Comparing the transition probabilities of $w$ and $\tilde{w}$, we have that 
    \begin{align*} \mbE[f(w_{k-1}(v_k))] &= \sum_{\substack{\mathbf{s} \in \{1,2\}^{k-1}\\ \mathbf{s} \text{ is valid}}} \Pr[\mathbf{a}=\mathbf{s}] \cdot f(w_{k-1}^{\mathbf{s}}(v_k)) \\ &= \sum_{\substack{\mathbf{s} \in \{1,2\}^{k-1}\\ \mathbf{s} \text{ is valid}}} \exp\left(- \sum_{i : v_i \in N_L(v_k)} w_{i}^{\mathbf{s}}(v_i) \right) \cdot \Pr[\tilde{\mathbf{a}}= \mathbf{s}] \cdot f(w_{k-1}^{\mathbf{s}}(v_k))\\
    &=\sum_{\substack{\mathbf{s} \in \{1,2\}^{k-1}\\ \mathbf{s} \text{ is valid}}} \exp\left(- \sum_{i : v_i \in N_L(v_k)} w_{k-1}^{\mathbf{s}}(v_i) \right) \cdot \Pr[\tilde{\mathbf{a}}= \mathbf{s}] \cdot f(w_{k-1}^{\mathbf{s}}(v_k))\\ &=  \mathbb{E}[f(\tilde{w}_{k-1}(v_k))e^{-X}].
    \end{align*}
\end{claimproof}

Our next claim is analogous to Claim~2.3 in~\cite{martinsson}. In~\cite{martinsson}, it was shown that for every $v_i\in N_L(v_k)$ the expectation of $\tilde{w}_t(v_i)$ is constant when $i$ is fixed and $t$ ranges from $1$ to $k-1$. In the modified process here this is still true, except we might set the weight $\tilde{w}(v_i)$ to $0$ if it is too large at time $i$. For us it only matters that the expected weight does not increase, so we prove the following.

\begin{claim} \label{claim:submartingale}
    For every $v_i\in N_L(v_k)$ and every $t\in [k-1]$, we have that $\mathbb{E}[\tilde{w}_t(v_i)]\le \alpha.$
\end{claim}
\begin{claimproof}
    We prove this by induction on $t$. Note that $\tilde{w}_0(v_i) = \alpha,$ by definition. Consider a step $1 \le t < i$. If $v_i \not\in N_R(v_t),$ then $\tilde{w}_{t}(v_i) = \tilde{w}_{t-1}(v_i).$ If $v_i \in N_R(v_t),$ then observe that $v_t \not\in N_L(v_k),$ as otherwise $v_t, v_i, v_k$ form a triangle. Therefore, $\mbE[\tilde{w}_t(v_i)] = (1 - e^{-\tilde{w}_t(v_t)}) \cdot 0 + e^{-\tilde{w}_t(v_t)} \cdot (e^{\tilde{w}_t(v_t)} \tilde{w}_{t-1}(v_i)) = \tilde{w}_{t-1}(v_i),$ where we remark that $\tilde{w}_t(v_t)$ corresponds to the weight of $v_t$ after potentially setting it to $0$ at the beginning of step $t$ and this is the weight used to update the weight of $v_i$. At step $t=i$, we have $\tilde{w}_i(v_i) = \tilde{w}_{i-1}(v_i)$ or $\tilde{w}_i(v_i) = 0$. Finally, the weight $\tilde{w}(v_i)$ does not change in steps $t > i,$ finishing the proof.
\end{claimproof}

The following claim is an immediate consequence of Claim~\ref{claim:submartingale} and linearity of expectation.
\begin{claim} \label{claim:expectation-of-X}
    \[ \mbE[X] \le \alpha \cdot d. \]
\end{claim}
\begin{claimproof}
    Indeed, $\mbE[X] = \sum_{v_i \in N_L(v_k)} \mbE[\tilde{w}_{k-1}(v_i)] \le \alpha \cdot |N_L(v_k)| \le \alpha \cdot d$.
\end{claimproof}
\begin{claim} \label{claim:w_tilde}
    \[ \tilde{w}_{k-1}(v_k) = \alpha e^X. \]
\end{claim}
\begin{claimproof}
    For any $i \in N_L(v_k),$ the weight of $v_k$ gets multiplied by $e^{\tilde{w}_i(v_i)} = e^{\tilde{w}_{k-1}(v_i)}$ and otherwise stays the same. Therefore, $\tilde{w}_{k-1}(v_k) = \alpha \cdot \exp\left(\sum_{v_i \in N_L(v_k)} \tilde{w}_{k-1}(v_i)\right) = \alpha e^X.$
\end{claimproof}

\begin{claim}
    \[ \Pr[v_k \in I] \ge \alpha / 4.\]
\end{claim}
\begin{claimproof}
    Define a function $f \colon [0, \infty) \rightarrow [0, \infty)$ as 
    \[ f(x) = 
    \begin{cases}
        0, &\text{if } x > \alpha \cdot e^{2 \alpha d} \\
        1 - e^{-x}, &\text{otherwise}.
    \end{cases}
    \]

    Note that $f(0)=0$ and hence
    \[ \Pr[v_k \in I] = \mbE[f(w_{k-1}(v_k))] = \mbE[f(\tilde{w}_{k-1}(v_k)) e^{-X}] = \mbE[f(\alpha e^X) \cdot e^{-X}], \]
    by Claims~\ref{claim:f(w)}~and~\ref{claim:w_tilde}.

    Note that if $X \le 2 \alpha d,$ then $\alpha e^X\le \alpha \cdot e^{2 \alpha d} \le \frac{\log d}{4 d^{1/2}}\le 1.$ Therefore, in that case, it holds that $e^{- \alpha e^X} \le 1 - \alpha e^X / 2,$ and thus $f(\alpha e^X) \cdot e^{-X} = (1 - e^{-\alpha e^X}) \cdot e^{-X} \ge \alpha e^X / 2 \cdot e^{-X}=
    \alpha / 2.$

    Finally, recalling that $\mbE[X] \le \alpha d$ by Claim~\ref{claim:expectation-of-X} and using Markov's inequality, we have
    \begin{align*}
         \Pr[v_k \in I] &= \mbE[f(\alpha e^X) \cdot e^{-X}] \ge \Pr[X \le 2 \alpha d] \cdot \mbE[f(\alpha e^X) e^{-X} \vert X \le 2 \alpha d]\\
         &\ge 1/2 \cdot \alpha / 2 = \alpha / 4.
    \end{align*}
\end{claimproof}
The previous claim proves Property~\ref{prop:large-prob}. Finally, observe that, regardless of the previous choices, in step $k$, $v_k$ is included in $I$ with probability at most $1 - e^{- \alpha e^{2 \alpha d}} \le \alpha e^{2 \alpha d},$ thus showing Property~\ref{prop:prob-not-too-big}. Finally, Property~c) of the lemma is a direct consequence of the definition of the random process we use to create $\mathcal{D}$.
\end{proof}

In order to prove our result on coloring graphs with sparse left neighborhoods (Theorem~\ref{thm:sparse}), we need a generalization of Lemma~\ref{lem:main} to this more general setting. The formal statement is captured by the following corollary, which we obtain by bootstrapping Lemma~\ref{lem:main} via some appropriate subsampling.

\begin{corollary} \label{cor:sparseleft}
    Let $d\in \mathbb{R}_{\ge 1}$ and let $G$ be a $d$-degenerate graph. Let $v_1,\ldots,v_n$ be a linear ordering witnessing that $G$ is $d$-degenerate and let $f\in (2,d^2)$ be such that $|E(G[N_L(v_i)])|\le d^2/f$ for every $i \in [n]$. Finally, let $0<\alpha<\frac{\log f}{8\sqrt{f}}$ be a parameter. Then there exists a probability distribution $\mathcal{D}$ on the independent sets of $G$ satisfying the following.

    \begin{enumerate}[label=\alph*)]
        \item \label{prop:large-prob2} For every $i \in [n], \, \Pr_{I \sim \calD}[v_i \in I] \ge \frac{\alpha\sqrt{f}}{32 d}.$
        \item \label{prop:prob-not-too-big2} For every $i \in [n]$ and every independent set $S \subseteq \{v_1, \dots, v_{i-1}\},$ it holds that
        \[ \Pr_{I \sim \calD}[v_i \in I \, \vert \, I \cap \{v_1, \dots, v_{i-1}\} = S] \le \frac{\alpha\sqrt{f} \cdot e^{2 \alpha \sqrt{f}}}{2d}. \]
    \end{enumerate}
    \end{corollary}
    \begin{proof}
    Set $p:=\frac{\sqrt{f}}{2d}\in (0,1)$ and let $U\subseteq V(G)$ be a random subset of vertices of $G$ created according to the following process: First, we create a random subset $V\subseteq V(G)$ in which every vertex of $G$ is included independently with probability exactly $p$. Then, a vertex $v_i$ is put into $U$ if and only if $v_i\in V$, $|N_L(v_i)\cap V|\le \sqrt{f}$ and $N_L(v_i)\cap V$ is an independent set in $G$. 

    Let $i\in [n]$ be arbitrary and let us estimate the probability that $v_i\in U$. Note that by Markov's inequality, we have that
    $$\mathbb{P}[|N_L(v_i)\cap V|>\sqrt{f}]<\frac{p|N_L(v_i)|}{\sqrt{f}}\le \frac{
pd}{\sqrt{f}}=\frac{1}{2}.$$
Similarly, we have
$$\mathbb{P}[N_L(v_i)\cap V \text{ not independent}]\le \mathbb{E}[|E(G[N_L(v_i)\cap V])|]=p^2|E(G[N_L(v_i)])|\le p^2\frac{d^2}{f}=\frac{1}{4}.$$
From this, we immediately conclude that $\mathbb{P}[v_i\in U]\ge \mathbb{P}[v_i\in V](1-\frac{1}{2}-\frac{1}{4})=\frac{p}{4}$ for every $i\in [n]$. 
    
    Note that by definition of the process to generate $U$, we will always have that the induced subgraph $G[U]$ is triangle-free and $d':=\sqrt{f}$-degenerate (the degeneracy is witnessed by the linear ordering obtained from $v_1,\ldots,v_n$ by omitting the vertices in $V(G)\setminus U$). Since $0<\alpha<\frac{\log f}{8\sqrt{f}}= \frac{\log d'}{4d'}$, we can apply Lemma~\ref{lem:main} to $G[U]$. Let $\mathcal{D}_U$ be the probability distribution on independent sets in $G[U]$ guaranteed by the lemma. Fixing $U$, from the properties guaranteed by the lemma, we then have
    $$\mathbb{P}_{I\sim \mathcal{D}_U}[v_i\in I]\ge \frac{\alpha}{4}$$ for every $v_i\in U$, and for every independent set $S\subseteq \{v_1,\ldots,v_{i-1}\}\cap U$ in $G$, it holds that
    $$\mathbb{P}_{I\sim \mathcal{D}_U}[v_i\in I|I\cap \{v_1,\ldots,v_{i-1}\}=S]\le \alpha\cdot e^{2\alpha d'}=\alpha\cdot e^{2\alpha\sqrt{f}}.$$

    Finally, let us consider the distribution $\mathcal{D}$ on independent sets of $G$ described by first sampling $U$ randomly according to the above process and then sampling an independent set from $\mathcal{D}_U$. We claim that it satisfies the desired properties. Indeed, for every $i\in [n]$, we have that
    $$\mathbb{P}_{I\sim \mathcal{D}}[v_i\in I]=\sum_{A\subseteq V(G):v_i\in A}\mathbb{P}_{I\sim \mathcal{D}}[v_i\in I|U=A]\cdot\mathbb{P}[U=A]=\sum_{A\subseteq V(G):v_i\in A}\mathbb{P}_{I\sim \mathcal{D}_A}[v_i\in I]\cdot\mathbb{P}[U=A]$$
    $$\ge \frac{\alpha}{4}\sum_{A\subseteq V(G): v_i\in A}\mathbb{P}[U=A]=\frac{\alpha}{4}\mathbb{P}[v_i\in U]\ge \frac{\alpha}{4}\cdot \frac{p}{4}=\frac{\alpha\sqrt{f}}{32d},$$ as desired. Next, consider any fixed independent set $S\subseteq \{v_1,\ldots,v_{i-1}\}$ in $G$. 
    We then have
    $$\mathbb{P}_{I\sim \mathcal{D}}[v_i\in I|I\cap \{v_1,\ldots,v_{i-1}\}=S]$$ $$=\sum_{\substack{A\subseteq V(G):\\ v_i\in A, S\subseteq A}}\mathbb{P}_{I\sim \mathcal{D}}[v_i\in I|I\cap \{v_1,\ldots,v_{i-1}\}=S, U=A]\cdot\mathbb{P}[U=A|I\cap \{v_1,\ldots,v_{i-1}\}=S]$$
    $$=\sum_{\substack{A\subseteq V(G):\\ v_i\in A, S\subseteq A}}\mathbb{P}_{I\sim \mathcal{D}_A}[v_i\in I|I\cap \{v_1,\ldots,v_{i-1}\}=S]\cdot\mathbb{P}[U=A|I\cap \{v_1,\ldots,v_{i-1}\}=S]$$
    $$\le \alpha e^{2\alpha\sqrt{f}}\sum_{\substack{A\subseteq V(G):\\ v_i\in A, S\subseteq A}}\mathbb{P}[U=A|I\cap \{v_1,\ldots,v_{i-1}\}=S]$$
    $$=\alpha e^{2\alpha\sqrt{f}} \mathbb{P}[v_i\in U|I\cap \{v_1,\ldots,v_{i-1}\}=S]$$
    $$\le \alpha e^{2\alpha\sqrt{f}} \mathbb{P}[v_i\in V|I\cap \{v_1,\ldots,v_{i-1}\}=S].$$
    Finally, note that the events $v_i\in V$ and $I\cap \{v_1,\ldots,v_{i-1}\}=S$ are independent. This is because, by definition of $V$, the event $v_i\in V$ is independent from $V\cap \{v_1,\ldots,v_{i-1}\}$, which by definition fully determines the set $U\cap \{v_1,\ldots,v_{i-1}\}$. Finally, by Lemma~\ref{lem:main},~\ref{prop:indepofwhatcomesafter} the latter set fully determines the distribution of $I\cap \{v_1,\ldots,v_{i-1}\}$ when $I$ is sampled from $\mathcal{D}_U$. Hence, we find that $$\mathbb{P}_{I\sim \mathcal{D}}[v_i\in I|I\cap\{v_1,\ldots,v_{i-1}\}=S]\le \alpha e^{2\alpha\sqrt{f}}\mathbb{P}[v_i\in V]=\alpha e^{2\alpha\sqrt{f}}p=\frac{\alpha\sqrt{f} e^{2\alpha\sqrt{f}}}{2d},$$ as desired. This concludes the proof of the corollary.
    \end{proof}
With Corollary~\ref{cor:sparseleft} at hand, we are now ready to present our proof of Theorem~\ref{thm:sparse}. Recall that Theorem~\ref{thm:sparse} also contains Theorem~\ref{thm:main2} as a special case, so in the following we will prove both of these theorems simultaneously. 
\begin{proof}[Proof of Theorem~\ref{thm:sparse}]
    We shall prove the statement with $c = 10^{-9}$ and $C = 10^4,$ using induction on $d$. Note that if $d \le e^{500}$ or $f\le e^{C}$, then the simple bound $\chi(G) \le d+1$ suffices to establish the claimed upper bound on $\chi(G)$, settling the induction basis.
    
    Moving on suppose $d> e^{500}$ and $f>e^{C}$, and let $G$ be an $n$-vertex $d$-degenerate graph, where $n \le 2^{cd}.$ Furthermore, let $v_1, \dots, v_n$ be a $d$-degenerate ordering of $G$ such that $N_L(v_i)$ spans at most $d^2/f$ edges in $G$ for every $i\in [n]$. Set $\ell = \frac{Cd}{3 \min\{\log f,\log(d/\log n)\}},$  $\beta = 0.1 \min\left\{\log f,\log\left(\frac{d}{\log n}\right)\right\}$ and $\alpha = \beta / \sqrt{f}.$ Observe that $\alpha < \frac{\log f}{8 \sqrt{f}}$ and $\beta>0.1\log\left(\frac{10^9}{\log 2}\right)>2$ (recall that we have $f>e^C$, so we know that $f$ is sufficiently large).
    We construct a sequence of induced subgraphs of $G$ as follows. Let $G_1 = G$. For $j = 1, \dots, \ell,$ consider the distribution $\calD_j$ on independent sets of $G_j$ given by Corollary~\ref{cor:sparseleft} applied with parameters $d$, $f$ and $\alpha$, and the $d$-degenerate ordering obtained from the sequence $v_1, \dots, v_n$ by only keeping the vertices in $V(G_j)$. Sample $I_j \sim \calD_j$ and let $G_{j+1} = G_j - I_j$.

    We will show that with positive probability, the graph $G_{\ell+1}$ is $(d/2)$-degenerate. Suppose this is the case. If we can show that $\chi(G_{\ell+1}) \le 2\ell,$ then we have $\chi(G) \le 3\ell = \frac{Cd}{\min\{\log f,\log (d / \log n)\}},$ as claimed. Next, we consider two cases. If $n \le 2^{cd/2},$ then by the induction hypothesis applied to $G_{\ell+1}$ with parameters $d/2$ and $f/4$ in place of $d$ and $f$, we obtain $\chi(G_{\ell+1}) \le \frac{C d / 2}{\min\{\log(f/4),\log (d / (2\log|V(G_{\ell+1})|))\}}$ $ \le \frac{2}{3} \frac{C d}{\min\{\log f, \log (d / \log n)\}} = 2\ell$ using that $f$ is large enough and $|V(G_{\ell+1})| \le n \le 2^{c d/2}$.
    On the other hand, if $n > 2^{cd / 2},$ we use the  simple bound $\chi(G_{\ell+1}) \le d/2 + 1,$ which implies $\chi(G_{\ell+1})\le 2\ell$ since then $2\ell \ge \frac{2 C d}{3\log (d / \log n)} \ge \frac{20000\cdot d}{3\log (2 \cdot 10^9)} > d.$

    Finally, it remains to prove that with positive probability, the graph $G_{\ell+1}$ is $(d/2)$-degenerate. It is enough to prove that for any fixed $k \in [n],$ we have
    \[ \Pr[ |N_L(v_k) \cap V(G_{\ell+1})| > d/2 ] < 1 / n. \]
    Indeed, then by the union bound, with positive probability the ordering obtained from $(v_1, \dots, v_n)$ by keeping the vertices in $V(G_{\ell+1})$ is a $(d/2)$-degenerate ordering of $G_{\ell+1}$. 

    So, fix $k \in [n]$. For $j \in [\ell+1],$ let $Y_j = |N_L(v_k) \cap V(G_j)|$ and for $j \in [\ell]$, let $X_j = Y_j - Y_{j+1} = |N_L(v_k) \cap V(G_j) \cap I_j|.$
    In the following, let $y_1,\ldots,y_j$ be an arbitrary but fixed sequence of integers with $y_{j}\ge d/2$.
    
    Observe that for every $j\in [\ell]$, by Corollary~\ref{cor:sparseleft}, Property~\ref{prop:large-prob2} we have 
    \[ \mbE[X_j|Y_1=y_1,\ldots,Y_{j}=y_{j}] \ge \frac{\alpha\sqrt{f}}{32d}\cdot\frac{d}{2} = \frac{\beta}{64}. \]
    Using concentration inequalities, we wish to show that $\Pr[\sum_{j=1}^\ell X_j < d/2] < 1/n,$ which would imply $\Pr[|N_L(v_k) \cap V(G_{\ell+1})| > d/2] < 1/n.$ In order to apply Azuma's inequality, we need to truncate the random variables $X_j$ to make them bounded without decreasing their expectation too much. Additionally, we need to consider the possibility that $Y_j \le d/2,$ in which case the conditional expectation $\mbE[X_j|Y_1=y'_1,\ldots,Y_{j}=y'_{j}]$ might be small. Of course, if $Y_j \le d/2,$ we are already done, so it is simply a matter of redefining the value of $X_j$ appropriately so we can formally apply Azuma's inequality.

    We start by showing that only a small proportion of the expectation of $X_j$ comes from the event when $X_j$ is larger than $\tau \coloneqq 4 \beta e^{2\beta}$. Note that by Corollary~\ref{cor:sparseleft}, Property~\ref{prop:prob-not-too-big2}, $X_j$, conditioned on $Y_1=y_1,\ldots,Y_j=y_j$, is stochastically dominated from above by $\Bin\left(d, \frac{\beta}{d} e^{2 \beta}\right)$. This implies that
    \begin{align*} &\mbE[X_j \vert X_j \ge \tau, Y_1=y_1,\ldots,Y_j=y_j] \cdot \Pr[X_j \ge \tau|Y_1=y_1,\ldots,Y_j=y_j]\\ &\le \mathbb{E}\left[\Bin\left(d, \frac{\beta}{d} e^{2 \beta}\right)\bigg\vert \Bin\left(d, \frac{\beta}{d} e^{2 \beta}\right)\ge \tau\right] \cdot \mathbb{P}\left[\Bin\left(d, \frac{\beta}{d} e^{2 \beta}\right)\ge \tau\right]\\ &= \lceil\tau\rceil \cdot \Pr\left[\Bin\left(d, \frac{\beta}{d} e^{2\beta}\right) \ge \tau\right] + \sum_{k = \lceil\tau\rceil + 1}^{d} \Pr\left[\Bin\left(d, \frac{\beta}{d} e^{2\beta}\right) \ge k\right]\\ &\le \lceil\tau\rceil \cdot \exp(- \beta e^{2\beta}) + \sum_{k=\lceil\tau\rceil+1}^\infty \exp( -(k - \beta e^{2\beta})/ 3) \le 10^{-10},
    \end{align*} 
    where we used the standard Chernoff bound $\Pr[\Bin(n, p) \ge np + \lambda] \le \exp(-\lambda / 3)$ valid for $\lambda \ge np$, and that $\beta>2$. This implies that
    $$\mbE[X_j \vert X_j \le \tau,Y_1=y_1,\ldots,Y_j=y_j] \cdot \Pr[X_j \le \tau|Y_1=y_1,\ldots,Y_j=y_j] \ge \frac{\beta}{64}-10^{-10}\ge  \frac{\beta}{65}.$$
    We are ready to define our truncated random variables. For $j \in [\ell],$ we define a random variable $Z_j$ as follows.
    \[ Z_j = \begin{cases} 
        \beta, &\text{if } Y_j \le d/2,\\
        X_j, &\text{if } Y_j > d/2 \text{ and } X_j \le \tau,\\
        0, &\text{otherwise.}        
    \end{cases}
    \]

    Observe that by definition, $0\le Z_j\le \tau$  for every $j\in [\ell]$. Note that each of the random variables $Z_1,\ldots,Z_{j-1}$ is determined by the values of $Y_1,\ldots,Y_j$. Hence, it follows that for every $j\in [n]$ and every sequence $z_1,\ldots,z_{j-1}$ of integers we have (denoting by $\mathcal{Z}$ the event $Z_1=z_1,\ldots,Z_{j-1}=z_{j-1}$):
    \begin{align*}
        &\mathbb{E}[Z_j|\mathcal{Z}]=\beta\cdot \mathbb{P}\left[Y_j\le \frac{d}{2}\bigg\vert\mathcal{Z}\right]+\mathbb{E}\left[X_j\bigg\vert X_j\le \tau, \mathcal{Z},Y_j>\frac{d}{2}\right]\cdot \mathbb{P}\left[X_j\le \tau\bigg\vert \mathcal{Z},Y_j>\frac{d}{2}\right]\cdot \mathbb{P}\left[Y_j>\frac{d}{2}\bigg\vert\mathcal{Z}\right]\\
        &\ge \beta\cdot \mathbb{P}\left[Y_j\le \frac{d}{2}\bigg\vert\mathcal{Z}\right]+\frac{\beta}{65}\cdot \mathbb{P}\left[Y_j> \frac{d}{2}\bigg\vert\mathcal{Z}\right]\ge \frac{\beta}{65}.
    \end{align*}

    Finally, this implies that the sequence $(W_i)_{i=0}^{\ell}$ of random variables, defined as $W_i:=-\frac{\beta}{65}i+\sum_{j=1}^{i}Z_j$, forms a submartingale. Further, $|W_i-W_{i-1}|=|Z_i-\frac{\beta}{65}|\le \tau$ for every $i$, and so by Azuma's inequality, we have
    $$\mathbb{P}\left[\sum_{j=1}^{\ell}Z_j\le \frac{\beta}{65}\ell-\eps\right]=\mathbb{P}\left[W_\ell-W_0\le -\eps\right]\le \exp\left(-\frac{\eps^2}{2\ell\tau^2}\right)$$ for every $\eps>0$. 

    Next, note that whenever we have $Y_{\ell+1}>\frac{d}{2}$, then, since $Y_{\ell+1}=Y_1-\sum_{j=1}^{\ell}X_j\le d-\sum_{j=1}^{\ell}X_j$, we must have $\sum_{j=1}^{\ell}{X_j}<\frac{d}{2}$. Furthermore, since then also $Y_1,\ldots,Y_\ell>\frac{d}{2}$, by definition we then have $Z_j\le X_j$ for every $j\in [\ell]$, implying $\sum_{j=1}^{\ell}{Z_j}<\frac{d}{2}$. Altogether, we conclude
    
    \[ \Pr[ |N_L(v_k) \cap V(G_{\ell+1})| > d/2 ]\le \Pr\left[\sum_{j=1}^\ell Z_j < d/2\right]. \]

    Note that our choice of constants implies that $\frac{\beta}{65}\ell=\frac{C\cdot 0.1}{3\cdot 65}d>3d$. Hence, setting $\varepsilon:=\frac{\beta}{65}\ell-\frac{d}{2}>\frac{\beta}{100}\ell$ in the above Azuma bound, we conclude:

    \begin{align*}
        \Pr\left[\sum_{j=1}^\ell Z_j < d/2\right] &\le \exp\left(- \frac{\varepsilon^2}{2\ell \tau^2}\right) <\exp\left(-\frac{\beta^2\ell}{2\cdot 10^4\cdot \tau^2}\right) = \exp\left( \frac{-\ell}{32 \cdot10^4\cdot  e^{4\beta}} \right)\\
        &\le \exp\left(\frac{-d^{0.6} \log^{0.4} n}{96 \log(d / \log n)}\right)\le \exp\left(-\frac{(10^9/\log 2)^{0.6}}{96\cdot \log(10^9/\log 2)}\log n\right) < 1/n,
    \end{align*}
    where in the last inequality we used that $n \le 2^{10^{-9}d}$. This yields the desired probability bound, concluding the proof.
    
\end{proof}

\section{On-line chromatic number of graphs with bounded clique number}\label{sec:online}
In this section, we provide the proofs of Proposition~\ref{prop:obs}, Theorem~\ref{thm:jacob} and, as announced in the introduction, provide a proof of the lower bound $g_r(n)\ge \Omega(n^{(r-2)/(r-1))})$ for the on-line chromatic number of $K_r$-free graphs on $n$ vertices (Proposition~\ref{prop:Kr-free-game} below). 

Note that we can equivalently define the on-line chromatic number for a hereditary family of graphs $\calG$ using the following game. The game is played by two players: Builder and Painter. Initially, both players are given a positive integer $n$, representing the total number of rounds the game is played. In round $i$, Builder presents a new vertex $v_i$ and connects it to some of the previous vertices $v_1, \dots, v_{i-1}$ such that the resulting graph $G[\{ v_1, \dots, v_i\}]$ is in $\calG$. Then, Painter colors vertex $v_i$ using a color from a color palette $\cC_n$ such that $v_i$ gets a different color from all of its neighbors. Painter wins if they can color all $n$ vertices, while Builder wins otherwise. The on-line chromatic number for $n$-vertex graphs of $\calG$ is the minimum size of the palette $\cC_n$ such that Painter has a winning strategy. Now, we present the proof of Proposition~\ref{prop:obs} which was our motivation for studying the on-line chromatic number of $K_r$-free graphs.

We recall its statement for the reader's convenience.

\propobsrestate*
\begin{proof}
    First, observe that in the game-theoretic formulation of $g_r(N)$, there is a strategy for Builder such that they never connect $v_i$ to two vertices of the same color since doing so leaves the same set of colors to be used by Painter while potentially reducing future options for Builder. Thus, we may assume that before Painter uses $d+1$ colors, with each new vertex Builder adds at most $d$ incident edges to previous vertices. 

    Fix a strategy for Builder that shows $g_r(N) \ge d+1$ and in which Builder connects each vertex $v_i$ to at most $d$ previous vertices, as discussed above. Now, we produce a graph $\Gamma$ which, informally speaking, contains all possible outcomes of the game. Without loss of generality, we may assume that in the first $N-1$ rounds, Painter uses colors from $[d]$. Note that the outcome of the game in steps $1, \dots, i$ is determined by the sequence of colors $(c_1, \dots, c_i) \in [d]^i$ used by Painter. For $i\ge 0,$ we say a sequence $(c_1, \dots, c_i)$ is \emph{valid} if it produces a valid coloring when Builder plays their fixed strategy, where we use $\emptyset$ to denote the empty sequence. The set of vertices of $\Gamma$ will be $\{ w_\mathbf{c} \, \vert \, \mathbf{c} = (c_1, \dots, c_i) \text{ is valid and }0\le i\le N-1 \}.$ In particular, $w_{\emptyset}$ is a vertex of $\Gamma$. We shall call $\mathbf{c}$ the label of vertex $w_\mathbf{c}$. The set of edges of $\Gamma$ is as follows. Let $\mathbf{c} = (c_1, \dots, c_{i-1})$ for some $i \in [N-1]$. Consider the game defined by sequence $(c_1, \dots, c_{i-1})$ and suppose Builder's strategy connects the next vertex $v_i$ to some vertex $v_j$ (where $j < i$). Then, $\Gamma$ contains the edge $\{w_{(c_1, \dots, c_{j-1})}, w_{(c_1, \dots, c_{i-1})}\}$.

    Note that ordering the vertices of $\Gamma$ by the length of their labels produces a $d$-degenerate ordering, thus $\Gamma$ is $d$-degenerate. The first vertex of this ordering is $w_{\emptyset}$. Furthermore, observe that the only edges of $\Gamma$ are between two vertices whose labels are prefixes of each other. Moreover, for any valid sequence $\mathbf{c} = (c_1, \dots, c_i)$, the induced subgraph of $\Gamma$ on $\{ w_{\emptyset}, w_{(c_1)}, w_{(c_1, c_2)}, \dots, w_{(c_1, \dots, c_i)}\}$ is isomorphic to a graph potentially produced by Builder so it is $K_r$-free. Thus, $\Gamma$ is $K_r$-free as well. Additionally, we have $|V(\Gamma)| \le \sum_{i=0}^{N-1} d^i \le \frac{d^{N} - 1}{d-1}.$    

    Finally, it remains to show that $\chi(\Gamma) = d+1$. Clearly $\chi(\Gamma) \le d+1$ since $\Gamma$ is $d$-degenerate. Now, suppose, for the sake of contradiction, that $\chi(\Gamma) \le d$ so there is a proper coloring $\phi \colon V(\Gamma) \rightarrow [d]$ of $\Gamma$. We shall iteratively construct a valid sequence $(c_1, \dots, c_N)$ which is a contradiction since we fixed a winning strategy for Builder and thus no valid sequence of length $N$ exists. Clearly $\emptyset$ is a valid sequence. Now assume that $i \in [N]$ and we have already defined colors $c_1, \dots, c_{i-1}$ such that $\mathbf{c} = (c_1, \dots, c_{i-1})$ is valid. Then, we let the next color $c_i$ equal $\phi(w_\mathbf{c})$. It remains to show that $(c_1, \dots, c_i)$ is a valid sequence. Indeed, consider the outcome of the first $i-1$ rounds of the game if Painter uses the sequence $\mathbf{c}$ along with Builder's $i$-th move. Let $v_1, \dots, v_i$ be the vertices in the on-line game and let $G$ be the graph produced by Builder on these vertices. By definition, for any $j < i$ such that $v_jv_i \in E(G_i)$, $w_\mathbf{c}$ and $w_{(c_1, \dots, c_{j-1})}$ are joined by an edge in $\Gamma$. Since Painter colored $v_j$ by $c_j = \phi(w_{(c_1, \dots, c_{j-1})})$ and $\phi$ is a proper coloring of $\Gamma$, it follows that Painter may color $v_i$ by $c_i$. In other words, $(c_1, \dots, c_i)$ is a valid sequence, as needed.    
\end{proof}

We proceed to the proof of Theorem~\ref{thm:jacob} which gives the upper bound $g_r(n) = O_r\left(\frac{n \log^{(r-1)} n}{\log^{(r-2)} n}\right).$ The statement easily follows from the following proposition by a recursive procedure.  

\begin{proposition}\label{prop:main2}
Let $n\in \mathbb{N}$ be sufficiently large and let $t:=\lfloor 4n \log \log n / \log n\rfloor$. There exists an on-line algorithm that colors an input on-line graph $G$ on $n$ vertices with at most $2t$ colors so that each color class is independent or contained in the neighborhood of some vertex.
\end{proposition}
Before proving Proposition~\ref{prop:main2}, let us directly demonstrate how it can be used to prove Theorem~\ref{thm:jacob}. 
To do so, we shall also need the following simple observation.
\begin{observation} \label{obs:small-color-classes}
    If there is an on-line coloring algorithm $A$ for $n$-vertex graphs from a graph class $\calG$ using $x$ colors, then there is an on-line coloring algorithm $A'$ for the same set of graphs using $2x$ colors with the property that every color class has size at most $s \coloneqq n/x$.    
\end{observation}
\begin{proof}
    Indeed, to produce the algorithm $A',$ we simply run $A,$ but on-line split the color classes into color classes of size $s$. That is, for a fixed color $c$, the first $s$ vertices colored $c$ by $A$ will be colored $(c,1)$ by $A'$, the next $s$ will be colored $(c, 2)$ by $A'$, and so on. If $A$ produces a coloring with color class sizes $m_1, \dots, m_x$, then $A'$ uses $\sum_{i=1}^x \ceil{m_i/s} \le \sum_{i=1}^x (m_i/s + 1) = 2x$ colors.
\end{proof}

\begin{proof}[Proof of Theorem~\ref{thm:jacob} given Proposition~\ref{prop:main2}]
    We proceed by induction on $r$. For each $r \ge 2,$ we shall define an on-line coloring algorithm $A_r = A_r(n)$ which given $n$, on-line colors $n$-vertex $K_r$-free on-line graphs using $\frac{C_r n \log^{(r-1)}n}{\log^{(r-2)} n}$ colors for some constant $C_r$ depending only on $r$.  In the base case $r=2$, we can simply color every vertex the same color, using one color. Now let $r \ge 3.$ By Proposition~\ref{prop:main2} and Observation~\ref{obs:small-color-classes}, there is an on-line coloring algorithm $A$ using at most $16 n \log \log n / \log n$ colors such that every color class is an independent set or contained in a neighborhood of a vertex and has size at most $s := \log n  / (8 \log \log n)$.

    Let $A_{r-1} = A_{r-1}(s)$ be the on-line coloring algorithm for $s$-vertex $K_{r-1}$-free graphs using at most $$\frac{C_{r-1} s \log^{(r-2)}s}{\log^{(r-3)} s}$$ colors given by the induction hypothesis. Now, $A_r = A_r(n)$ is the algorithm obtained by running $A_{r-1}$ on each color class produced by $A$, each time using a new, disjoint set of colors. Note that this is feasible, since every color class produced by $A$ must span a $K_{r-1}$-free subgraph. Observe that $s$ is only a function of $n$ so the algorithm $A_{r-1}$ is well-defined at the start and thus so is $A_r$. We remark that a color class given by $A$ might have size smaller than $s$, but crucially, not larger than $s$, so each run of the algorithm $A_{r-1}$ will produce a valid coloring. For $n$ large enough, the total number of colors used by $A_r$ is at most
    \[ \frac{16 n\log \log n}{\log n} \cdot \frac{C_{r-1} s \log^{(r-2)} s}{\log^{(r-3)} s} \le \frac{16 n\log \log n}{\log n} \cdot C_{r-1} \frac{\log n}{8\log\log n} \cdot \frac{2\log^{(r-1)} n}{\log^{(r-2)} n} \le \frac{C_r n \log^{(r-1)} n}{\log^{(r-2)} n},
    \]
    for an appropriate choice of $C_r$.    
\end{proof}

We now supplement the proof of Proposition~\ref{prop:main2}.
\begin{proof}[Proof of Proposition \ref{prop:main2}]

Recall that we assume $n$ to be sufficiently large and $t=\lfloor 4n \log \log n / \log n\rfloor$. Consider the following setup.  We give ourselves two disjoint sets, each with $t$ colors, to use for the on-line coloring process. One
of these sets we call $L$, and the other set is $[t]=\{1,\ldots,t\}$. The vertices of each color in $L$ will form an
independent set, while the vertices of color $c \in [t]$ will share a
common neighbor.  Our goal is thus to present such an on-line coloring algorithm.

At step $h$, vertex $v_h$ is added. For $j \in [t]$, we let $C_{j,h}$ denote
the subset of the first $h$ vertices which have been assigned color $j$, and we define
$C_{j,0}$ as the empty set for each $j$. 

At step $h$ we will define sets $N_{j,h}$ for each color $j \in [t]$ which
satisfy the following properties: 

\begin{enumerate}
\item Every vertex in $N_{j,h}$ is assigned a color in $L$.
\item Every vertex in $N_{j,h}$ is adjacent to every vertex in $C_{j,h}$. 
\item $|N_{j,h}| \geq \left(\frac{t}{2n}\right)^{|C_{j,h}|-1}\frac{t}{2}$
if $C_{j,h}$ is nonempty.
\item The sets $N_{j,h}$ with $j \in [t]$ are pairwise disjoint. 
\end{enumerate}

We start by setting $N_{j,0}$ to be the empty set for each $j \in [t]$ and note that the desired properties are satisfied for $h=0$.  

We next describe the coloring strategy. At step $h$, the $h$-th vertex $v_h$ is added.
If it has no neighbor of color $\ell$ for some $\ell \in L$, then we
assign to $v_h$ color $\ell$ and continue to the next step (setting
$N_{j,h}=N_{j,h-1}$ for each $j \in [t]$). Otherwise, $v_h$ has at
least $t$ neighbors with color in~$L$, one of each color in~$L$. If there
exists a $j \in [t]$ with $C_{j,h-1}$ nonempty and $v_h$ has at least
$(\frac{t}{2n})|N_{j,h-1}|$ neighbors in $N_{j,h-1}$, then we select the smallest such $j$, assign
to $v_h$ color $j$ (i.e, $C_{j,h}=C_{j,h-1}\cup \{v_h\}$) and set $N_{j,h}$ to be the set of neighbors of
$v_h$ in $N_{j,h-1}$. We set $N_{j',h}=N_{j',h-1}$ for $j' \not = j$. If no
such $j$ exists, then the number of neighbors of vertex $v_h$ in at least one
of the $N_{j,h-1}$ is at most $\sum_{j} \frac{t}{2n}|N_{j,h-1}| \leq t/2$.
Hence, $v_h$ has at least $t-t/2=t/2$ neighbors of color $L$ not in any of the sets
$N_{j,h-1}$. It suffices then to show that there is an unused color in
$[t]$, i.e., that one of the $C_{j,h-1}$ is empty. Indeed, we could then give
vertex $v_h$ the color $j$, resulting in $C_{j,h}=\{v_j\}$, let $N_{j,h}$ be the set of neighbors of $v_h$ of color $L$
that are not contained in any of the sets $N_{j,h-1}$, and set $N_{j',h}=N_{j',h-1}$ for $j' \not = j$, preserving all the invariants. Note that, indeed, $|N_{j,h}| \geq t/2=(t/(2n))^{|C_{j,h}|-1} t/2$.

The
number of colors $j$ for which $|C_{j,h-1}| \geq 2n/t$ is less than
$\frac{t}{2}$ since  $h-1 < n$ vertices haven been presented thus far. Each color $j$
with $1 \leq |C_{j,h-1}| < 2n/t$ satisfies $|N_{j,h-1}| \geq
\left(\frac{t}{2n}\right)^{2n/t}t/2 \geq 2n/t$, and as the sets $N_{j,h-1}$
are disjoint and their union has size at most $h-1<n$, there are less than
$\frac{t}{2}$ such colors $j$. Hence, there is at least one color $j\in [t]$
not used so far (i.e. $C_{j,h-1}=\emptyset$), and we can assign vertex $h$ that color and continue. 

The coloring procedure above maintains all the invariants. In the end, we
get a coloring of the vertices with the desired properties. 
\end{proof}
Having finished the proof of our new upper bound on $g_r(n)$, we next move on to proving the lower bound for $g_r(n)$ announced in the introduction, which we restate here for convenience.
\begin{proposition}
    \label{prop:Kr-free-game}
    For each $r \geq 3$ and $d \geq 1$, the following holds. In the on-line coloring game where the graph is restricted to be $K_r$-free, Builder can force $d$ colors in $O_r(d^{1 + \frac{1}{r - 2}})$ rounds. In other words, $g_r(n) = \Omega_r(n^{\frac{r-2}{r-1}})$.
\end{proposition}
We prove the following stronger lemma.
\begin{lemma}
For each $r \geq 2$, there exist constants $c_r > 0$ such that the following holds. For each $y \geq 1$, in the on-line coloring game where the graph is $K_r$-free, in at most $y^{r-1}$ rounds, Builder can construct a set of vertices $I$ such that
\begin{enumerate}
    \item $I$ is independent.
    \item Each color appears at most $y$ times in $I$.
    \item $I$ has size at least $c_r y^{r - 1}$.
\end{enumerate}
\end{lemma}
Before moving on to the proof, let us note that by setting $y = (d / c_r)^{\frac{1}{r - 2}}$ the lemma immediately implies Proposition~\ref{prop:Kr-free-game}: By (2) and (3) there are at least $c_r y^{r - 2} = d$ different colors, and Builder takes at most $y^{r - 1} = O_r(d^{\frac{r - 1}{r - 2}})$ rounds in total. We now prove the lemma.
\begin{proof}
    We argue by induction on $r$. The base case $r = 2$ is easy: Builder adds $y$ independent vertices and lets $I$ be the entire vertex set. The requirements are satisfied with $c_2 = 1$.

    Suppose the result holds for $(r - 1)$. We describe Builder's strategy for $K_r$-free graphs, with $c_r := c_{r - 1} 2^{-r}$. Builder maintains a set $J$ with the following properties.
    \begin{enumerate}
        \item $J$ is independent.
        \item Each color appears at most $y$ times in $J$.
    \end{enumerate}
    Initially, $J = \emptyset$. Each iteration, let $\Gamma$ be the set of colors that appear at least $y/2$ times in $J$.
    \begin{enumerate}
        \item If $|\Gamma| \geq c_{r - 1} 2^{-r+1} y^{r - 2}$, then Builder obtains $I$ by taking the set of vertices in $J$ with color in $\Gamma$. The set $I$ satisfies (1) and (2) by the properties of $J$, and we have
        $$|I| \geq |\Gamma| \cdot \frac{y}{2} \geq c_{r - 1} 2^{-r+1} y^{r - 2} \cdot \frac{y}{2} = c_r y^{r - 1}$$
        as desired.
        \item Otherwise, let $U \subset J$ be a subset of $J$ containing one vertex of each color in $\Gamma$. Builder runs the strategy for $K_{r - 1}$-free graphs with parameter $y' = y / 2$, adding additional edges between each new vertex and all vertices in $U$. As $U$ is independent, Builder never creates a $K_r$. By the induction hypothesis, Builder can construct an independent set $V$ with size at least $c_{r - 1} (y/2)^{r - 2}$ in at most $(y / 2)^{r - 2}$ rounds, such that each color appears at most $y / 2$ times in $V$. 
        
        At this point, Builder replaces $J$ with $J' = (J \setminus U) \cup V$. We now check that $J'$ satisfies both required properties. As every vertex in $V$ is not adjacent to any of the vertices in $(J \setminus U)$, $J'$ is an independent set. Furthermore, consider a color $c$. If $c \in \Gamma$, then $c$ doesn't appear in $V$; otherwise, $c$ appears at most $y/2$ times in $J$ and at most $y/2$ times in $V$. In both cases, $c$ appears at most $y$ times in $J'$. 

        Finally, we check that 
        $$|J'| = |J| - |U| + |V| \geq |J| - c_{r - 1} 2^{-r+1} y^{r - 2} + c_{r - 1} (y/2)^{r - 2} \geq |J| + c_{r - 1} 2^{-r+1} y^{r - 2}.$$
        So we have $|J'| \geq |J| + 2c_r y^{r - 2}$.
    \end{enumerate}
    In conclusion, in each iteration Builder either obtains the desired $I$ or increases the size of $J$ by at least $2c_r y^{r - 2}$ in at most $(y/2)^{r - 2}$ rounds. After $(y / 2)$ iterations, either Builder has obtained the desired $I$, or the size of $J$ is at least $c_r y^{r - 1}$, in which case $J$ itself is the desired $I$. We conclude that Builder can always construct the desired $I$ in at most $(y/2)^{r - 1}$ rounds.
\end{proof}




\section{Harris' conjecture for higher clique number}
Given the main theorem of \cite{martinsson} and its many applications, it is natural to seek a generalization of Harris' conjecture for $K_r$-free graphs.
\begin{problem}
    \label{problem:harris-kr}
    What is the maximum fractional chromatic number of a $d$-degenerate, $K_r$-free graph?
\end{problem}
The result of Johansson mentioned in the introduction shows that the chromatic number of $K_r$-free graphs with maximum degree $\Delta$ is at most $O_r\left(\frac{\Delta \log\log \Delta}{\log \Delta}\right)$. One might guess  that the answer to \Cref{problem:harris-kr} is also $O_r\left(\frac{d \log\log d}{\log d}\right)$. In this section, we show that surprisingly this is not the case, by establishing Theorem~\ref{thm:lower} as announced in the introduction. We repeat its statement here for convenience.
\thmlowerrestate*
In the following we describe and analyze the construction used to prove Theorem~\ref{thm:lower}, which is inspired by the well-known Zykov construction~\cite{zykov} of triangle-free graphs with large chromatic number. As input, the construction takes two numbers $x,d\in\mathbb{N}$ and a graph $\Gamma$ which is $x$-degenerate.
 We then iteratively construct a sequence of graphs $G_{\Gamma, n}$. Set $G_{\Gamma, 1}$ as the one-vertex graph $K_1$. We construct $G_{\Gamma, n + 1}$ from $G_{\Gamma, n}$ as follows.
\begin{enumerate}
    \item Take $d$ disjoint copies $G_{\Gamma, n}^{(1)}, \cdots, G_{\Gamma, n}^{(d)}$ of $G_{\Gamma, n}$.
    \item For each $d$-tuple of vertices $C \in V(G_{\Gamma, n}^{(1)}) \times \cdots \times V(G_{\Gamma, n}^{(d)})$, we add a disjoint copy $\Gamma_C$ of $\Gamma$, and add edges between every vertex of $C$ and every vertex of $\Gamma_C$.
\end{enumerate}
\textbf{Remark: }For $\Gamma$ the one-vertex graph, this is very similar to Zykov's construction.

We can easily verify the following properties of $G_{\Gamma, n}$ by induction on $n$.
\begin{enumerate}
    \item $G_{\Gamma, n}$ is $(d + x)$-degenerate.
    \item We have $\omega(G_{\Gamma, n}) \leq \omega(\Gamma) + 1$. Indeed, a maximal clique with at least three vertices in $G_{\Gamma, n + 1}$ either lies entirely inside one copy of $G_{\Gamma, n}$, or consists of one vertex from a copy of $G_{\Gamma, n}$ and a clique in $\Gamma$.
\end{enumerate}
We now bound the fractional chromatic number of the graphs in the sequence $(G_{\Gamma, n})_{n\ge 1}$. The crucial observation is the following recursive relation.
\begin{lemma}
\label{lem:fractional-recursion}
For every $n\in \mathbb{N}$ it holds that
$$\frac{1}{\chi_f(G_{\Gamma, n + 1})} \leq \frac{1}{d + 1} \left(\frac{d}{\chi_f(G_{\Gamma, n})} + \frac{1}{\chi_f(\Gamma)}\left(1 - \frac{1}{\chi_f(G_{\Gamma, n})}\right)^d\right).$$
\end{lemma}
Before giving the proof of this key inequality, it will be convenient to recall equivalent definitions of the fractional chromatic number. First of all, it is not too difficult to see that the definition of $\chi_f(G)$ (in terms of probability distributions on independent sets) that we gave in the introduction is equivalent to defining $\chi_f(G)$ as the optimal value of the following linear program (here $\mathcal{I}(G)$ denotes the collection of all independent sets in $G$):
\begin{align*}
\tag{P}
    \text{min} \sum_{I \in \mathcal{I}(G)}&{x_I} \\
    \text{s.t.} \sum_{I\in \mathcal{I}(G): v \in I}&{x_I}\ge 1~~(\forall v \in V(G)), \\
    & x_I \ge 0~~(\forall I \in \mathcal{I}(G)).
\end{align*}
By linear duality, it follows that $\chi_f(G)$ also equals the optimal value of the following dual program:\begin{align*}
\tag{D}
    \text{max} &\sum_{v\in V(G)}{w_v} \\
    \text{s.t.}~~~~&\sum_{v \in I}{w_v}\le 1~~(\forall I \in \mathcal{I}(G)), \\
    &~~~~~w_v \ge 0~~(\forall v\in V(G)).
\end{align*}
From this formulation (by using an appropriate scaling), it can be seen that for any real number $r\in \mathbb{R}_{\ge 1}$ the following holds: A graph $G$ satisfies $\chi_f(G)\ge r$ if and only if there exists a weight function $w:V(G)\rightarrow \mathbb{R}_{\ge 0}$ with $\sum_{v\in V(G)}{w(v)}=1$ such that $w(I):=\sum_{v\in I}{w(v)}\le \frac{1}{r}$ for every independent set $I$ in $G$. The latter characterization is the one which will turn out most convenient in the proof of Lemma~\ref{lem:fractional-recursion}, which we now supply.\begin{proof}[Proof of Lemma~\ref{lem:fractional-recursion}]
    By the previously discussed characterization of the fractional chromatic number, there exists a weight function $w_n: V(G_{\Gamma, n}) \to \RR_{\geq 0}$ such that $\sum_{v \in V(G_{\Gamma, n})} w_n(v) = 1$ and $w_n(I) \leq \frac{1}{\chi_f(G_{\Gamma, n})}$ for each independent set $I$ of $G_{\Gamma, n}$. There also exists a weight function $w_{\Gamma}: V(\Gamma) \to \RR_{\geq 0}$ such that $\sum_{v \in V(\Gamma)} w_\Gamma(v) = 1$ and $w_\Gamma(J) \leq \frac{1}{\chi_f(\Gamma)}$ for each independent set $J$ of $\Gamma$. We now consider the weight function $w_{n + 1}: V(G_{\Gamma, n + 1}) \to \RR_{\geq 0}$ defined by
    $$w_{n + 1}(v) = \begin{cases}
    w_n(v), \text{ if } v \text{ lies in a copy of $G_{\Gamma, n}$,} \\
    w_\Gamma(v) \prod_{u \in C} w_n(u), \text{ if }v \text{ lies in $\Gamma_C$ for some $C\in V(G_{\Gamma,n}^{(1)})\times\cdots\times V(G_{\Gamma,n}^{(d)})$}.
    \end{cases}$$
    Then we have
    $$\sum_{v \in V(G_{\Gamma, n + 1})} w_{n + 1}(v) = \sum_{i = 1}^d \underbrace{\sum_{v \in V(G_{\Gamma, n}^{(i)})} w_n(v)}_{=1} + \sum_{C \in V(G^{(1)}_{\Gamma, n}) \times \cdots \times V(G^{(d)}_{\Gamma, n})} \underbrace{\sum_{v \in \Gamma_C} w_\Gamma(v)}_{=1} \prod_{u \in C} w_n(u)$$
    $$=d+\left(\sum_{v\in V(G_{\Gamma,n})}w_n(v)\right)^d=d+1.$$
    Furthermore, let $I$ be any independent set in $G_{\Gamma, n + 1}$. Let $I_i = I \cap V(G_{\Gamma, n}^{(i)})$, and let $x_i = \sum_{v \in I_i} w_n(v)$. Then $I$ can only intersect $\Gamma_C$ with $C \in \Bar{I_1} \times \cdots \times \Bar{I_d}$, and for such $C$ we have
    $$\sum_{v \in I \cap V(\Gamma_C)} w_{n + 1}(v) = \sum_{v \in I \cap V(\Gamma_C)} w_\Gamma(v) \prod_{u \in C} w_n(u) \leq \frac{1}{\chi_f(\Gamma)} \prod_{u \in C} w_n(u).$$
    Thus, we obtain
    $$\sum_{v \in I} w_{n + 1}(v) \leq \sum_{i = 1}^{d} x_i + \sum_{C \in \Bar{I_1} \times \cdots \times \Bar{I_d}} \frac{1}{\chi_f(\Gamma)} \prod_{u \in C} w_n(u) $$ $$= \sum_{i=1}^{d}x_i+\frac{1}{\chi_f(\Gamma)}\prod_{i=1}^{d} \left(\sum_{u\in \bar{I_i}}w_n(u)\right)=\sum_{i=1}^{d}x_i+\frac{1}{\chi_f(\Gamma)}\prod_{i=1}^{d}\left(\sum_{u\in V(G_{\Gamma,n}^{(i)})}w_n(u)-\sum_{u\in I_i}w_n(u)\right)$$ $$= \sum_{i = 1}^d x_i +  \frac{1}{\chi_f(\Gamma)} \prod_{i = 1}^d (1 - x_i).$$
    Regarded as a function in $\mathbf{x} \in [0, 1]^d$, the right hand side is an increasing linear function in each $x_i$. As $0 \leq x_i \leq \frac{1}{\chi_f(G_{\Gamma, n})}$ for each $i$, we get
    $$\sum_{v \in I} w_{n + 1}(v) \leq d \frac{1}{\chi_f(G_{\Gamma, n})} + \frac{1}{\chi_f(\Gamma)} \left(1 - \frac{1}{\chi_f(G_{\Gamma, n})}\right)^d$$
    which gives
    $$\frac{\sum_{v \in I} w_{n + 1}(v)}{\sum_{v \in V(G_{\Gamma, n + 1})} w_{n + 1}(v)} \leq \frac{1}{d + 1}\left(d \frac{1}{\chi_f(G_{\Gamma, n})} + \frac{1}{\chi_f(\Gamma)} \left(1 - \frac{1}{\chi_f(G_{\Gamma, n})}\right)^d\right).$$
    Thus we conclude that
    $$\frac{1}{\chi_f(G_{\Gamma, n + 1})} \leq \frac{1}{d + 1} \left(\frac{d}{\chi_f(G_{\Gamma, n})} + \frac{1}{\chi_f(\Gamma)}\left(1 - \frac{1}{\chi_f(G_{\Gamma, n})}\right)^d\right)$$
    as desired.
\end{proof}
Since $G_{\Gamma, n}$ is isomorphic to a subgraph of $G_{\Gamma, n  + 1}$, the sequence $(\chi_f(G_{\Gamma, n}))_{n = 1}^\infty$ is non-decreasing in $n$. Furthermore, it is clearly bounded from above (namely by $d+x+1$, since all the graphs in the sequence $(G_{
\Gamma,n})_{n\ge 1}$ are $(d+x)$-degenerate). Therefore, the limit $\chi \coloneqq \lim_{n \to \infty} \chi_f(G_{\Gamma, n})$ exists. Taking $n \to \infty$ in \Cref{lem:fractional-recursion}, we have
$$\frac{1}{\chi} \leq \frac{1}{d + 1} \left(\frac{d}{\chi} + \frac{1}{\chi_f(\Gamma)}\left(1 - \frac{1}{\chi}\right)^d\right).$$
Simplifying, we have
$$\frac{1}{\chi} \leq \frac{1}{\chi_f(\Gamma)}\left(1 - \frac{1}{\chi}\right)^d.$$
Letting $\frac{d}{\chi} = \alpha$, we get
$$\alpha e^{\alpha} \leq \frac{d}{\chi_f(\Gamma)}.$$
Since $\log (a+3) >1$ for all $a>0$, we conclude that
$$\frac{d}{\chi} < \log\left(\frac{d}{\chi_f(\Gamma)} + 3\right).$$
We have proven the following.
\begin{lemma}
\label{lem:recursion}
Let $d, x \geq 2$ be integers. Suppose $\Gamma$ is an $x$-degenerate, $K_r$-free graph. Then there exists a $(d + x)$-degenerate, $K_{r + 1}$-free graph $G_\Gamma$, such that
$$\frac{d}{\chi_f(G_\Gamma)} \leq \log\left(\frac{d}{\chi_f(\Gamma)} + 3\right).$$
\end{lemma}
Finally, let $f_r(d)$ denote the infimum of $\frac{d}{\chi_f(G_\Gamma)}$ over all $d$-degenerate, $K_r$-free graphs. Taking $x = d$ in \Cref{lem:recursion} gives
$$f_{r + 1}(2d) \leq 2 \log(f_r(d) + 3).$$
As $f_2(d) \leq d$, we get $f_r(d) \lesssim_r \log^{(r - 2)}(d)$ by induction on $r$, establishing \Cref{thm:lower}. 



\section{Proofs of the Applications}\label{sec:applications}
In this section we give the proofs of the applications of Theorem~\ref{thm:main2} presented at the end of the introduction. We start with our two applications to Hadwiger's conjecture. For this, we first derive a result about so-called \emph{clustered colorings} of $K_t$-minor-free graphs from known results in the literature. Given a graph $G$ and integers $k$ and $d$, a \emph{$d$-clustered $k$-coloring} of $G$ is any assignment $c:V(G)\rightarrow [k]$ of one of $k$ colors to the vertices such that each monochromatic component, i.e., each component of an induced subgraph of the form $G[c^{-1}(i)]$ for some $i\in [k]$, has at most $d$ vertices. There is a rich and interconnected theory on clustered colorings of various sparse classes of graphs, we refer to the survey~\cite{wood} by Wood for an extensive summary of known results.  To prove Corollaries~\ref{cor:had1}~and~\ref{cor:had2}, we will use the following clustered coloring result which follows by combining results from a recent paper of Dujmovi\'{c} et al.~with the recent breakthrough by Gorsky, Seweryn and Wiederrecht~\cite{gowi} on polynomial bounds for the Graph Minor Structure Theorem originally due to Robertson and Seymour~\cite{RS}.
\begin{lemma}\label{lem:Ktdecomposition}
There exists a polynomial $p(\cdot)$ such that the following holds. For every $\Delta, t\in \mathbb{N}$, every $K_t$-minor-free graph with maximum degree at most $\Delta$ admits a $p(\Delta t)$-clustered $3$-coloring.
\end{lemma}
To obtain Lemma~\ref{lem:Ktdecomposition}, we start by stating the following formulation of the graph minor structure theorem. As we shall not need them here, we refrain from giving the detailed definitions of the notions appearing in the statement and instead refer the interested reader to~\cite{dujmovic}.

\begin{theorem}[Graph Minor Structure Theorem~\cite{RS}, cf.~Theorem~8 in~\cite{dujmovic}]\label{thm:graphminors}
For every graph $H$, there exists $k_H\in \mathbb{N}_0$ such that every $H$-minor-free graph $G$ admits a tree decomposition $(B_x|x\in V(T))$ such that the torso $G_x$ of $B_x$ is $k_H$-almost
embeddable for each vertex $x\in V(T)$.
\end{theorem}
The aforementioned breakthrough of Gorsky, Seweryn and Wiederrecht~\cite{gowi} shows that Theorem~\ref{thm:graphminors} holds with $k_H=q(|V(H)|)$, where $q(\cdot)$ is an explicit polynomial.

Dujmovi\'{c} et al.~\cite{dujmovic} proved the following theorem. While Theorem~\ref{thm:partition} below is not explicitly stated in their paper as in the following, one can easily verify that the proofs of Theorems~10,~19 and~20 in~\cite{dujmovic} indeed directly yield the following statement.

\begin{theorem}\label{thm:partition}
Let $H$ be a graph, and let $k_H\in \mathbb{N}$ be such that the statement of Theorem~\ref{thm:graphminors} holds. Then for every $\Delta\in\mathbb{N}$, every $H$-minor-free graph of maximum degree at most $\Delta$ admits an $O(k_H^{18}\Delta^5)$-clustered $3$-coloring. 
\end{theorem}

The statement of Lemma~\ref{lem:Ktdecomposition} can now immediately be obtained from Theorem~\ref{thm:partition} by using that Theorem~\ref{thm:graphminors} holds with $k_H=q(|V(H)|)$, where $q(\cdot)$ is a polynomial.

With Lemma~\ref{lem:Ktdecomposition} at hand, we can now give the proof of Corollary~\ref{cor:had1}. We will make use of the following classical result due to Kostochka~\cite{kostochka} and Thomason~\cite{thomason}. 

\begin{theorem}\label{thm:kostthomas}
There exists an absolute constant $A>0$ such that the following holds for $t\ge 2$. 
Every $K_t$-minor-free graph has average degree at most $At\sqrt{\log t}$. In particular, every $K_t$-minor-free graph is $At\sqrt{\log t}$-degenerate. 
\end{theorem}
 
    \begin{proof}[Proof of Corollary~\ref{cor:had1}]
As announced after the corollary, we will in fact prove that the statement claimed in the corollary even holds in the more general setting when we exclude any almost bipartite graph $H$ (this includes $H=K_3$ as a special case). So let $\varepsilon>0$ and an almost bipartite graph $H$ be fixed. By taking the constant $K$ to be sufficiently large in the statement of the corollary, we may assume that $t$ is sufficiently large for all the following estimates to hold, as otherwise the statement is trivially true. Now, consider any $H$-free and $K_t$-minor-free graph $G$ such that $\Delta:=\Delta(G)\le e^{t^{1-\varepsilon}}$. By Theorem~\ref{thm:kostthomas}, we have that $G$ is $d$-degenerate, where $d:= At\sqrt{\log t}$. Let $k\in \mathbb{N}$ be minimal such that $H$ is isomorphic to a subgraph of $K_{1,k,k}$. Then every neighborhood in~$G$ induces a $K_{k,k}$-free subgraph. 

Let $D\in\mathbb{N}$ be the degree of the polynomial $p(\cdot)$ in the statement of Lemma~\ref{lem:Ktdecomposition} (this is an absolute constant). We then know from the lemma that there is a $3$-coloring of the vertices of $G$ where every monochromatic component has order at most $O((\Delta t)^D).$ Let $V_1,V_2,V_3$ be the three color classes of this coloring and consider any monochromatic component $F$, i.e., any connected component of $G[V_i]$ for some $i\in \{1,2,3\}$. Clearly, $F$ inherits from $G$ the property that it is $d$-degenerate and that its neighborhoods induce $K_{k,k}$-free subgraphs.

Then considering an ordering $v_1,\ldots,v_{|V(F)|}$ of the vertices of $F$ witnessing its $d$-degeneracy, the K\"{o}v\'{a}ri-S\'{o}s-Tur\'{a}n theorem implies that $|E(F[N_L(v_j)])|\le O(d^{2-1/k})$ for every $j\in [n]$. In particular, we have $|E(F[N_L(v_j)])|\le d^2/f$ for every $j\in [n]$, where $f:=d^{1/(2k)}$ (using here that $t$ is chosen sufficiently large). For $t$ sufficiently large, we furthermore have 
$$|V(F)|\le O((\Delta t)^D)= O(\exp(Dt^{1-\varepsilon})t^D)\le e^{d^{1-\varepsilon/2}}\le 2^{cd},$$ where $c>0$ is the absolute constant from Theorem~\ref{thm:sparse}. All in all, this implies that we may apply Theorem~\ref{thm:sparse} to the graph $F$ with parameters $f$ and $d$, yielding that
$$\chi(F)\le \frac{C d}{\min\{\log  f,\log(d/\log |V(F)|)\}}\le \frac{Cd}{\min\{(\log d)/2k,\log(d/d^{1-\varepsilon/2})\}}$$ $$\le \frac{C}{\min\{1/(2k),\varepsilon/2\}}\frac{d}{\log d}\le \frac{C}{\min\{1/(2k),\varepsilon/2\}}\frac{At\sqrt{\log t}}{\log t}\le \frac{K}{3}\frac{t}{\sqrt{\log t}},$$ where $K:=\frac{3AC}{\min\{1/(2k),\varepsilon/2\}}$. All in all, for $i\in \{1,2,3\}$ we have shown that every connected component of $G[V_i]$ has chromatic number at most $\frac{K}{3}\frac{t}{\sqrt{\log t}}$, so it also follows that $\chi(G[V_i])\le \frac{K}{3}\frac{t}{\sqrt{\log t}}$ for every $i\in \{1,2,3\}$. 
Together, this yields
$\chi(G)\le \chi(G[V_1])+\chi(G[V_2])+\chi(G[V_3])\le \frac{Kt}{\sqrt{\log t}}$, as desired. This concludes the proof of the corollary as well as its extension from triangle-free graphs to $H$-free graphs, where $H$ is almost bipartite.
  
    \end{proof}

Before we are able to present the proof of Corollary~\ref{cor:had2}, we first deduce a (qualitative) strengthening of Lemma~\ref{lem:Ktdecomposition} concerning graphs with a forbidden complete bipartite subgraph (rather than graphs of bounded maximum degree). To do so, we make use of a result of Ossona de Mendez, Oum and Wood~\cite{demendez}. In the following, given a graph $H$ we say that a ($\le 1$)-subdivision of $H$ is any graph obtained from $H$ by replacing a subset of the edges of $H$ with internally vertex disjoint paths of length two.

\begin{theorem}[consequence of~Theorem~2.3 in~\cite{demendez}]\label{thm:demendez}
Let $a,b\in \mathbb{N}$, $a>2$, let $G$ be a graph with no $K_{a,b}$-subgraph and let $r\in \mathbb{R}_{\ge 1}$ be such that every graph $H$ of which some ($\le 1$)-subdivision is a subgraph of $G$ has average degree at most $r$. Then the vertex set of $G$ can be partitioned into sets $V_1,\ldots,V_a$ such that $\Delta(G[V_i])\le \Delta$ for every $i\in [a]$, where $$\Delta:=(r-a)\left(\binom{\lfloor r\rfloor}{a-1}(b-1)+r/2\right)+r=O(r^ab).$$
\end{theorem}
With this result at hand, we can obtain the following variant of Lemma~\ref{lem:Ktdecomposition} when excluding a complete bipartite subgraph.
Lemma~\ref{lem:Ktdecomposition} can be recovered as the special case $a=1$ and $b=\Delta+1$.
\begin{lemma}\label{lem:noKab}
There exists a polynomial $P(\cdot)$ such that the following holds. Let $a,b,t\in \mathbb{N}$ and let $G$ be a $K_t$-minor-free graph with no $K_{a,b}$-subgraph. Then $G$ admits a $P(t^ab)$-clustered $3a$-coloring.
\end{lemma}
\begin{proof}
Possibly after suitably increasing parameters we may w.l.o.g.~assume $a>2$ and that $t$ is sufficiently large. Let $p(\cdot)$ be the polynomial from Lemma~\ref{lem:Ktdecomposition}. Let $r:=At\sqrt{\log t}$, where $A>0$ is the constant from Theorem~\ref{thm:kostthomas}, and note that $r\le t^2$ (since we assumed $t$ to be sufficiently large). Note that any graph $H$ of which some ($\le 1$)-subdivision is a subgraph of $G$ must be a minor of $G$ and thus also $K_t$-minor-free. Hence, by Theorem~\ref{thm:kostthomas} every such graph $H$ must have average degree at most $r$. Hence, by Theorem~\ref{thm:demendez} we can partition the vertex set of $G$ into parts $V_1,\ldots,V_a$ such that $\Delta(G[V_i])\le \Delta$ for all $i\in [a]$, where $\Delta$ is defined as in the Theorem and satisfies $\Delta=O(r^ab)=O(t^{2a}b)$. By Lemma~\ref{lem:Ktdecomposition}, for every $i\in [a]$ the graph $G[V_i]$ admits a $p(\Delta t)$-clustered $3$-coloring. By using disjoint color sets on distinct sets $V_i$, this implies that $G$ admits a $p(\Delta t)$-clustered $3a$-coloring. Note that since $\Delta=O(t^{2a}b)$, we have that $p(\Delta t)$ is polynomially bounded in terms of $t^ab$. This concludes the proof of the lemma.
\end{proof}
Equipped with Lemma~\ref{lem:noKab}, we are now ready to present the proof of Corollary~\ref{cor:had2}.
    
\begin{proof}[Proof of Corollary~\ref{cor:had2}]
Since the statement uses asymptotic notation in $t$ throughout, we may w.l.o.g assume $t$ to be chosen sufficiently large in the following. Suppose that $G$ is a counterexample to Hadwiger's conjecture for parameter $t$, i.e., $G$ is $K_t$-minor-free and satisfies $\chi(G)\ge t$. Let $A>0$ denote the absolute constant from Theorem~\ref{thm:kostthomas} and $C>0$ the absolute constant from Theorem~\ref{thm:sparse}. Let $D\in \mathbb{N}$ be an absolute constant bigger than the degree of the polynomial $P(\cdot)$ in Lemma~\ref{lem:noKab}. Let $\varepsilon:\mathbb{N}\rightarrow (0,1)$ be any function chosen such that $\varepsilon(x)\rightarrow 0$ as well as $(\log x)^{\varepsilon(x)}\rightarrow \infty$ for $x\rightarrow \infty$. Set $$a:=\lfloor (\log t)^{1/2-\varepsilon(t)} \rfloor$$ and $$b:=\left\lfloor t^{-a}\exp\left(t/(De^{4aAC\sqrt{\log t}})\right)\right\rfloor,$$

and note that $a=(\log t)^{1/2-o(1)}$ and $b=e^{t^{1-o(1)}}$. Our goal will be to prove the following claim.

\medskip

\paragraph*{\textbf{Claim.}} Graph $G$ contains a $K_{a,b}$ subgraph or it has a subgraph $H'$ contained in the neighborhood of some vertex such that $|V(H')|\le At\sqrt{\log t}$ and $|E(H')|\ge \frac{ (At\sqrt{\log t})^2}{e^{4aAC\sqrt{\log t}}}=t^{2-o(1)}$.

\medskip

Before proving this claim, let us observe that it implies the statement of the corollary: If $G$ contains a $K_{a,b}$ subgraph, then it trivially satisfies the statement of the corollary. Next, suppose the second option of the claim occurs, i.e., $G$ has a subgraph $H'$ contained in the neighborhood of some vertex such that $|V(H')|\le At\sqrt{\log t}$ and $|E(H')|\ge \frac{ (At\sqrt{\log t})^2}{e^{4aAC\sqrt{\log t}}}=t^{2-o(1)}$. If $|V(H')|\le t$, this is already a subgraph with the properties as desired in the statement of the corollary. If $|V(H')|>t$, then taking a random subset $U$ of $V(H')$ of size $t$ and letting $H:=H'[U]$ yields a subgraph of order $t$ with at least $|E(H')|\cdot \frac{t(t-1)}{|V(H')|(|V(H')|-1)}\ge t^{2-o(1)}\cdot \frac{1}{O(\log t)}=t^{2-o(1)}$ edges in expectation. Hence, also in this case we find a subgraph on at most $t$ vertices and $t^{2-o(1)}$ edges contained in the neighborhood of some vertex. Hence, once we have proved the above claim, the statement of the corollary will also be proven.

To prove the above claim, suppose towards a contradiction $G$ neither contains $K_{a,b}$ as a subgraph, nor does it have a subgraph $H'$ with the asserted properties. Applying Lemma~\ref{lem:noKab} and using the choice of $D$ and our assumption that $t$ is sufficiently large, we find that $G$ admits a $(t^ab)^D$-clustered $3a$-coloring. Let $U_1,\ldots,U_{3a}$ denote the color classes of this coloring and consider any connected component $F$ of $G[U_i]$ for some $i\in [3a]$. Note that by Theorem~\ref{thm:kostthomas} we have that $F$ is $d$-degenerate, where $d:= At\sqrt{\log t}$. By the property of the clustered $3a$-coloring, we further have 
$$|V(F)|\le (t^ab)^D\le \left(\exp(t/(De^{4aAC\sqrt{\log t}}))\right)^D=\exp(t/e^{4aAC\sqrt{\log t}}).$$ 
In particular (using that $t$ is sufficiently large), we have $|V(F)|\le 2^{cd}$ where $c>0$ is the absolute constant from Theorem~\ref{thm:sparse}. Let us set $f:=e^{4aAC\sqrt{\log t}}$ and consider a linear ordering $v_1,\ldots,v_{|V(F)|}$ of the vertices of $F$ witnessing its $d$-degeneracy. By our assumption every subgraph of $G$ that is fully contained in a neighborhood of a vertex and has at most $At\sqrt{\log t}$ vertices has less than $\frac{ (At\sqrt{\log t})^2}{e^{4aAC\sqrt{\log t}}}=\frac{d^2}{f}$ edges. In particular, every left-neighborhood in the above linear ordering spans at most $d^2/f$ edges. We can now see that all conditions of Theorem~\ref{thm:sparse} are met for the graph $F$ with parameters $d$ and $f$. Thus, it follows that 
$$\chi(F)\le \frac{Cd}{\min\{\log f,\log(d/\log |V(F)|)\}}$$
$$=\frac{Cd}{4aAC\sqrt{\log t}}\le \frac{t}{4a}.$$ We have thus shown that every connected component of $G[U_i]$ is $\lfloor t/(4a)\rfloor$-colorable for every $i\in [3a]$, and thus also $\chi(G[U_i])\le \lfloor t/(4a)\rfloor$ for every $i$. It follows that $\chi(G)\le \sum_{i=1}^{3a}\chi(G[U_i])\le 3a\cdot \lfloor t/(4a)\rfloor\le \frac{3t}{4}<t$, a contradiction to our assumption that $\chi(G)\ge t$. This contradiction shows that our initial assumption was wrong, proving the claim. As previously explained, this also concludes the proof of the corollary.
\end{proof}
Next, we get to the proof of Corollary~\ref{cor:erd} about Erd\H{o}s's conjecture on graph Ramsey numbers. 
\begin{proof}[Proof of Corollary~\ref{cor:erd}]
Let $G$ be a graph such that $\chi(G)=k$ but $r(G)<r(K_k)$. Let $d$ denote the degeneracy of $G$, i.e. the minimum $d\in \mathbb{N}$ such that $G$ is $d$-degenerate. Then $G$ has a subgraph $H$ of minimum degree at least $d$. By a well-known probabilistic argument (see e.g.~\cite[Proposition~5.4.1.]{yuvalnotes}), it then holds that $r(H)>2^{d/2}$. Since it is furthermore well-known that $r(K_k)<4^k$, it follows that $2^{d/2}<r(H)\le r(G)<r(K_k)<4^k$ and thus $d<4k$. So $G$ is a $4k$-degenerate graph with chromatic number $\chi(G)=k$. Let $v_1,\ldots,v_n$ be a linear ordering of the vertices of $G$ witnessing its $4k$-degeneracy. Let $c,C>0$ be the absolute constants from Theorem~\ref{thm:sparse} and set $f:=e^{5C}$. If $n\ge \min\{e^{4k/e^{5C}},2^{4ck}\}$, then $G$ has $e^{\Omega(k)}$ vertices, concluding the proof in this case. Also, if there exists some $i\in [n]$ such that $|E(G[N_L(v_i)])|>\frac{(4k)^2}{f}$, then $G[N_L(v_i)]$ is a subgraph of $G$ with at most $4k$ vertices and $\Omega(k^2)$ edges, concluding the proof of the corollary also in this case. 

Thus, moving on, suppose that $n< \min\{e^{4k/e^{5C}},2^{4ck}\}$ and $|E(G[N_L(v_i)])| \leq \frac{(4k)^2}{f}$ for every $i\in [n]$. We are now in the position to apply Theorem~\ref{thm:sparse} to $G$ with parameters $d':=4k$ and $f$, and thus we find that
$$\chi(G)\le \frac{C(4k)}{\min\{\log f, \log((4k)/\log n)\}}=\frac{4k}{5}<k,$$ a contradiction since we assumed $\chi(G)=k$. This contradiction shows that we always either have that $|V(G)|\ge e^{\Omega(k)}$ or that $G$ has a subgraph on $O(k)$ vertices with $\Omega(k^2)$ edges, concluding the proof.
\end{proof}
Next, we present the short proofs of Corollaries~\ref{cor:logbound} and~\ref{cor:hereditary} on coloring hereditary graph classes.
\begin{proof}[Proof of Corollary~\ref{cor:logbound}]
Let $H$ be any given almost bipartite graph and let $\varepsilon>0$. Let $k\in \mathbb{N}$ be such that $H$ is isomorphic to a subgraph of $K_{1,k,k}$. Let $C,c>0$ be the absolute constants from Theorem~\ref{thm:sparse}. We shall prove the statement of the corollary with $K:=C\frac{1+\varepsilon}{\varepsilon}$ for sufficiently large $n$. Note that by suitably adjusting the constant $K$, it will then also hold for all $n$.

So let $G$ be an $H$-free $d$-degenerate graph on $n$ vertices. If $d<\log^{1+\varepsilon} n$, then we have $\chi(G)\le d+1\le \log^{1+\varepsilon}n+1\le K\max\left\{\log^{1+\varepsilon}n, \frac{d}{\log d}\right\}$, so the desired bound holds. Moving on, we may thus assume $d\ge \log^{1+\varepsilon} n$. 

Let $v_1,\ldots,v_n$ be a linear ordering of the vertices such that every vertex has at most $d$ left-neighbors. Since $G$ is $H$-free and $H$ is isomorphic to a subgraph of $K_{1,k,k}$, we have that $G[N_L(v_i)]$ is $K_{k,k}$-free. Hence, by the K\H{o}v\'{a}ri-S\'{o}s-Tur\'{a}n theorem we have $e(G[N_L(v_i)])\le (k+1)d^{2-1/k}\le d^{1-1/(2k)}$ for every $i\in [n]$, where we used that $n$ and thus $d$ is sufficiently large. Therefore, the number of edges spanned by any left-neighborhood is at most $d^2/f$, where $f:=d^{1/(2k)}$ for $n$ sufficiently large. Furthermore, since $d>\log^{1+\varepsilon}n$, we have $n\le e^{d^{1/(1+\varepsilon)}}\le 2^{cd}$ for $n$ sufficiently large. All in all, we have now met all the requirements for an application of Theorem~\ref{thm:sparse} to $G$ and thus obtain
$$\chi(G)\le \frac{Cd}{\min\{\log f,\log(d/\log n)\}}\le \frac{Cd}{\min\{\log( d^\varepsilon),\log(d/d^{1/(1+\varepsilon)})\}}=C\frac{1+\varepsilon}{\varepsilon}\frac{d}{\log d}$$ $$\le K\max\left\{\frac{d}{\log d},\log^{1+\varepsilon}n\right\},$$ concluding the proof.
\end{proof}
\begin{proof}[Proof of Corollary~\ref{cor:hereditary}]
    It follows immediately from our assumptions, using that $d(\cdot)$ is monotone, that every $n$-vertex graph in $\mathcal{G}$ is $d(n)$-degenerate. We furthermore claim that under the assumptions in the statement of the corollary, there exists some $k\in \mathbb{N}$ such that every graph $G\in \mathcal{G}$ has no $K_{k,k}$-subgraph. For this, it suffices to note that since $d(n)=o(n)$, there exists some $n_0\in \mathbb{N}$ such that for $n\ge n_0$ neither $K_n$ nor $K_{n,n}$ is contained in $\mathcal{G}$. Since $\mathcal{G}$ is hereditary, it follows that every graph $G\in \mathcal{G}$ has neither $K_{n_0}$ nor $K_{n_0,n_0}$ as an induced subgraph. But it is easy to see that this implies that every graph $G\in \mathcal{G}$ does not contain $H:=K_{R(n_0,n_0),R(n_0,n_0)}$ as a subgraph (here $R(n_0,n_0)$ denotes the diagonal Ramsey number). Since $H$ is bipartite, we can now apply Corollary~\ref{cor:logbound} to find that there exists a constant $K$ depending only on $\varepsilon$ and $H$ such that every graph $G\in \mathcal{G}$ which is $d$-degenerate satisfies $\chi(G)\le K\max\{\log^{1+\varepsilon/2} n,\frac{d}{\log d}\}$. Every $n$-vertex graph $G\in \mathcal{G}$ is $d(n)$-degenerate and $d(n)>\log^{1+\varepsilon} n$, and so we find that
    $$\chi(G)\le K\max\left\{\log^{1+\varepsilon/2}n, \frac{d(n)}{\log d(n)}\right\}=K\frac{d(n)}{\log d(n)}$$ for every sufficiently large $n$, concluding the proof of the corollary.
    \end{proof}
Finally, we also provide the proof of the improved upper bound for the chromatic number of $(g,k)$-planar graphs when excluding an almost bipartite subgraph. On top of the fact that $(g,k)$-planar graphs are $O(\sqrt{(g+1)(k+1)})$-degenerate~\cite{demendez,wood}, we will also need the following clustered coloring result due to Wood.
\begin{proposition}[cf.~Proposition~45,~\cite{wood}]\label{prop:12clust}
Every $(g,k)$-planar graph has a $O((g+1)^{9/2}(k+1)^{7/2})$-clustered $12$-coloring.
\end{proposition}
We are now ready for the proof of our last application, Corollary~\ref{cor:gk}. 
\begin{proof}[Proof of Corollary~\ref{cor:gk}]As announced after the corollary, we can in fact show a more general statement, namely that the claimed asymptotic upper bound on the chromatic number of $(g,k)$-planar graphs holds if we exclude any almost bipartite graph $H$ ($H=K_3$ is a special case). So let $H$ be a fixed almost bipartite graph, and let $k\in \mathbb{N}$ be minimal such that $H$ is isomorphic to a subgraph of $K_{1,k,k}$. As the claimed bound on the chromatic number is asymptotic in terms of $(g+1)(k+1)$, in the following we may assume that the product $(g+1)(k+1)$ is sufficiently large. Let $G$ be any given $H$-free $(g,k)$-planar graph. Then $G$ is $d$-degenerate, where $d= A\sqrt{(g+1)(k+1)}$ for some absolute constant $A>0$. Furthermore, by Proposition~\ref{prop:12clust} there exists a partition $(V_1,\ldots,V_{12})$ of the vertices of $G$ such that for every $i\in [12]$ all components of $G[V_i]$ have order at most $B(g+1)^{9/2}(k+1)^{7/2}$, where $B>0$ is some absolute constant. Now consider any component $F$ of $G[V_i]$. Clearly, $F$ is $d$-degenerate since $G$ is. Let $v_1,\ldots,v_{|V(F)|}$ be a linear ordering of the vertices witnessing $d$-degeneracy. Since we assumed that $(g+1)(k+1)$ is sufficiently large, we then have 
$$|V(F)|\le B(g+1)^{9/2}(k+1)^{7/2}\le e^{\sqrt{d}}\le 2^{cd},$$ where $c>0$ is the absolute constant from Theorem~\ref{thm:sparse}. Furthermore, since $F$ is $K_{1,k,k}$-free, the K\"{o}v\'{a}ri-S\'{o}s-Tur\'{a}n theorem implies that $|E(G[N_L(v_j)])|\le O(d^{1-1/k})\le d^{1-1/(2k)}$ for all $j\in [|V(F)|]$. Setting $f:=d^{1/(2k)}$, we may now apply Theorem~\ref{thm:sparse} to $F$ to obtain:
$$\chi(F)\le \frac{Cd}{\min\{\log f, \log(d/\log |V(F)|)\}}\le \frac{Cd}{\min\{1/(2k),1/2\}\log d}$$ $$=O_H(d/\log d)=O_H\left(\frac{\sqrt{(g+1)(k+1)}}{\log ((g+1)(k+1))}\right).$$ Since $F$ was an arbitrary component of $G[V_i]$, we also find that $\chi(G[V_i])\le O_H\left(\frac{\sqrt{(g+1)(k+1)}}{\log ((g+1)(k+1))}\right)$ for every $i\in [12]$ and thus
$$\chi(G)\le \sum_{i=1}^{12}{\chi(G[V_i])}\le O_H\left(\frac{\sqrt{(g+1)(k+1)}}{\log ((g+1)(k+1))}\right),$$ as desired. This concludes the proof of the corollary and its extension to $H$-free graphs when $H$ is almost bipartite.
\end{proof}
\section{Open problems}
We conclude by stating several intriguing open problems that are natural next steps to consider after our work. If resolved, these would come with further applications.

The perhaps most important open problem is whether we can further improve the exponential requirement on the number of vertices in Theorems~\ref{thm:main2} and~\ref{thm:sparse}.
\begin{problem}
Can the upper bound $2^{cd}$ on the number of vertices assumed in Theorems~\ref{thm:main2} and~\ref{thm:sparse} be weakened to $e^{\omega(d)}$ (with a corresponding improvement of the upper bound on $\chi$ when $n=e^{\omega(d)}$)? What if we additionally assume that the graph has large (constant) girth?
\end{problem}
We remark that even improving the value of the constant $c$ in Theorem~\ref{thm:main2} could have significant consequences: For example, showing that $c$ can be taken to be some moderately large constant, together with an appropriate bound on $\chi$, would solve Erd\H{o}s's aforementioned problem on graph Ramsey numbers for all triangle-free graphs, and perhaps even all $H$-free graphs where $H$ is almost bipartite when $k$ is large enough. The idea here would be to use the trivial bound $r(G)\ge |V(G)|$ on the Ramsey number, so that in the case that the graph is exponentially big in terms of $k$ (with a large constant in the exponent), one would immediately obtain $r(G)\ge |V(G)|\ge r(K_k)$. 

A second potential qualitative strengthening of our main results (Theorems~\ref{thm:main2} and~\ref{thm:sparse}) concerns the question whether our \emph{global} requirement (bounding the number of vertices) can be replaced by a \emph{local} requirement, namely bounding the maximum degree.
\begin{problem}
\label{conj:max-degree}
Is the following statement true for all $\varepsilon>0$? If $G$ is a $d$-degenerate triangle-free graph with maximum degree $\Delta(G)\le e^{d^{1-\varepsilon}}$, then we have $$\chi(G) = O\left(\frac{d}{\log d}\right).$$
\end{problem}
The proof of this or similar statements would likely require some technical innovation which manages to ``localize'' the random coloring process considered in this paper, such that the union bound we employ can be replaced by an application of the Lov\'{a}sz local lemma. Here is an example of an immediate application if the answer to the above problem is affirmative: Given a graph $H$, an \emph{$H$-immersion} is any graph obtained from $H$ by replacing its edges by edge disjoint (but not necessarily vertex disjoint) paths. In 1988, Lescure and Meyniel~\cite{lescure} conjectured an immersion analogue of Hadwiger's conjecture, namely that graphs without a $K_t$-immersion are $(t-1)$-colorable. This remains open, but an upper bound of $O(t)$ on the degeneracy (and thus chromatic number) is known~\cite{devos}. Van den Heuvel and Wood~\cite{heuvelwood} showed that the vertex set of every $K_t$-immersion-free graph can be split into two parts, each inducing a subgraph of maximum degree $O(t^3)$. From this one can see that an affirmative answer to Problem~\ref{conj:max-degree} would imply an upper bound $O(t/\log t)$ on the chromatic number of triangle-free graph without a $K_t$-immersion, thus proving a strong form of Lescure and Meyniel's conjecture for triangle-free graphs. This bound would also be optimal up to constant factors: Take a random $(t-2)$-regular graph and delete all vertices on triangles. This graph w.h.p.~has chromatic number $\Omega(t/\log t)$ but has too small maximum degree to host a $K_t$-immersion. 

The last two open problems we would like to advertise here are the intriguing questions whether $K_r$-free $d$-degenerate graphs with chromatic number $d+1$ must also be exponentially large, and (looking at our proof methods) the closely related question of whether one can give an asymptotic improvement of the trivial upper bound $d+1$ on the fractional chromatic number of $K_r$-free $d$-degenerate graphs when $r\ge 4$. 
\begin{problem}
    \label{conj:main-kr}
    Is it true that if the chromatic number of a $K_r$-free $d$-degenerate graph is $d + 1$, then the graph has $e^{\Omega_r(d)}$ vertices?
\end{problem}
\begin{problem}
    \label{conj:harris-kr}
    Let $r\ge 4$. Is it true that the maximum fractional chromatic number of $K_r$-free $d$-degenerate graphs is $o_r(d)$?
\end{problem}
Recall that our Theorem~\ref{thm:lower} gives a lower bound $\Omega(d/\log^{(r-2)}d)$ on the fractional chromatic number which gets closer and closer to linear in $d$ as $r$ gets bigger.

\medskip

\paragraph*{\textbf{Acknowledgments.} We would like to thank Noga Alon, Anders Martinsson and Jonathan Tidor for interesting discussions on the topic.}
\bibliographystyle{abbrvurl}
\bibliography{references}

\end{document}